\documentclass[12pt]{article}       
\title{Induced representations and harmonic analysis on finite groups}

\author{Fabio Scarabotti, Filippo Tolli}
     
\usepackage{amssymb}       
\usepackage{amsmath}       
\usepackage{latexsym}       
\usepackage{amsthm}       
\usepackage{epic,latexsym,amsbsy,amsmath,amscd,amssymb}

\usepackage{amsmath}  
\usepackage{amssymb} 
 
\renewcommand{\square}{\circ}

 \newtheorem{definition}{Definition} [section]       
 \newtheorem{remark}[definition]{Remark}

 \newtheorem{proposition}[definition]{Proposition}       
 \newtheorem{theorem}[definition]{Theorem}       
 \newtheorem{corollary}[definition]{Corollary}       
  \newtheorem{lemma}[definition]{Lemma}

\newcommand{\Ind}{\mbox{\rm Ind}}

\newcommand{\Hom}{\mbox{\rm Hom}}
\newcommand{\Res}{\mbox{\rm Res}}

\newcommand{\tr}{\mbox{\rm tr}}

\newcommand{\ran}{{\rm Ran}}

\def\TT{{\mathcal{T}}}

\def\TW{\overset{\wedge}{T}}
\def\LV{\overset{\vee}{L}}

\def\CC{{\mathbb{C}}}

\begin{document}
\maketitle

\begin{abstract}  
Given a finite group $G$ and a subgroup $K$, we study the commutant of $\text{Ind}_K^G\theta$, where $\theta$ is an irreducible $K$-representation. After a careful analysis of Frobenius reciprocity, we are able to introduce an orthogonal basis in such commutant and an associated Fourier transform. Then we translate our results in the corresponding Hecke algebra, an isomorphic algebra in the group algebra of $G$. Again a complete Fourier analysis is developed and, as particular cases, we obtain some results of Curtis and Fossum on the irreducible characters of Hecke algebras. Finally, we develop a theory of Gelfand-Tsetlin bases for Hecke algebras. 
\footnote{{\it AMS 2010 Math. Subj. Class.}Primary: 20C15. Secondary: 20C08, 43A90

\it Keywords: Induced representation, Frobenius reciprocity, Fourier transform, Hecke algebra, spherical function, Gelfand-Tsetlin basis}
\end{abstract}

%%%%%%%%%%%%%%%%%%%%%%%%%%%%%%%%%%%%%%%%%%%%%%%%%%%%%
%%%%%%%%%%%%%%%%%%%%%%%%%%%%%%%%%%%%%%%%%%%%%%%%%%%%%
\section{Introduction}
%%%%%%%%%%%%%%%%%%%%%%%%%%%%%%%%%%%%%%%%%%%%%%%%%%%%%
%%%%%%%%%%%%%%%%%%%%%%%%%%%%%%%%%%%%%%%%%%%%%%%%%%%%%
Let $G$ be a finite group and  denote by $\widehat{G}$ a complete set of irreducible, pairwise inequivalent unitary representations of $G$ and by $L(G)$ its group algebra. The dimension of $\sigma\in\widehat{G}$ is denoted by $d_\sigma$ and $M_{d,d}(\mathbb{C})$ is the algebra of all $d\times d$ complex matrices.
One of the key facts in the representation theory of $G$ is the isomorphism

\begin{equation}\label{introisom1}
L(G)\cong \bigoplus_{\sigma\in\widehat{G}}M_{d_\sigma,d_\sigma}(\mathbb{C}),
\end{equation} 
which is given explicitly by the Fourier transform; see \cite{book}, Section 9.5. Now suppose that $K$ is a subgroup of $G$, denote by $X=G/K$ the corresponding homogeneous space  and by $L(X)$ the permutation module of all complex valued functions defined on $X$. Then we have the isomorphism
\begin{equation}\label{introisom2}
\text{Hom}_G(L(X),L(X))\cong \bigoplus_{\sigma\in J}M_{m_\sigma,m_\sigma}(\mathbb{C}),
\end{equation} 
where $J$ is the set of all $\sigma \in \widehat{G}$  contained in $L(X)$ and $m_\sigma$ the multiplicity of $\sigma$ in  $L(X)$; see again \cite{book}, Section 9.4. Clearly, \eqref{introisom1} is a particular case of \eqref{introisom2}, because $\text{Hom}_G(L(G),L(G))\cong L(G)$. The (spherical) Fourier transform in the setting of \eqref{introisom2} has been extensively studied when $(G,K)$ is a Gelfand pair, that is when the algebra $\text{Hom}_G(L(X),L(X))$ is commutative, which is equivalent to say that  $L(X)$ is multiplicity free. This analysis is based on the thery of spherical functions. There are several accounts on this subject and on its many applications; see \cite{Bump, GelGeorg, book,Diaconis, Macdonald,Terras}. In \cite{st1,st2,st3} we extended the theory of spherical Fourier transforms to homogeneous spaces with multiplicity and gave several applications, mainly to probability and statistics (an earlier example may be found in \cite{ScarabottiLapBer}). In particular, we showed that  multiplicity freeness is not an essential tool in order to develop a satisfactory theory and to perform explicit calculations. \\

In the present paper we face a more general problem. Suppose that $\theta$ is an irreducible $K$-representation. Then we have again
\begin{equation}\label{introisom3}
\text{Hom}_G(\text{Ind}_K^G\theta,\text{Ind}_K^G\theta)\cong \bigoplus_{\sigma\in J}M_{m_\sigma,m_\sigma}(\mathbb{C}),
\end{equation} 
where $J$ is the set of all $\sigma \in \widehat{G}$  contained in $\text{Ind}_K^G\theta$ and $m_\sigma$ is  the multiplicity of $\sigma$ in  $\text{Ind}_K^G\theta$. We  introduce a Fourier transform that gives an explicit isomorphism for \eqref{introisom3}.
The irreducible characters of the  algebra $\text{Hom}_G(\text{Ind}_K^G\theta,\text{Ind}_K^G\theta)$ were computed by C.W. Curtis and T.V. Fossum in \cite{CURFOS}. Accounts of their theory may be found in \cite{CR2} and in the recent expository paper \cite{st4}, where it is presented as a generalization of the theory of spherical functions of finite Gelfand pairs. But the results of Curtis and Fossum can be used only for the Fourier analysis of functions in the {\em center} of the algebra. In our approach, a complete set of matrix coefficients are obtained and the results of Curtis and Fossum may be derived in a more transparent form.\\

The plan of the paper is the following. Section \ref{Secprel} is devoted to fix notation and to introduce one of the 
key ideas of the paper: suitable scalar products are used not only in the representation spaces but also in the space of intertwining operators (normalized Hilbert-Schmidt scalar products). This leads to several natural orthogonality relations: in Section \ref{SecorthFrob} these are obtained by a detailed analysis of Frobenius reciprocity. 
The results may be summarized in a commutative diagram of isomorphisms that are either isometries or multiples of isometries. In particular, for $\sigma\in J$, the {\em explicit} isomorphism between $\text{Hom}_K(\theta,\text{Res}^G_K\sigma)$ and $\text{Hom}_G(\sigma,\text{Ind}^G_K\theta)$ and a particular choice of an {\em orthonormal} basis in $\text{Hom}_K(\theta,\text{Res}^G_K\sigma)$ lead to an {\em explicit} decomposition of the $\sigma$-isotypic component in $\text{Ind}^G_K\theta$. In the particular case of Gelfand pairs, this corresponds to the choice of a $K$-invariant vector in each spherical representation and to the use of the spherical functions to decompose the permutation representation; see Section 4.6 in \cite{book}. In Section \ref{Seccommutant} the results on Frobenius reciprocity are used to get a natural orthogonal basis in $\text{Hom}_G(\text{Ind}_K^G\theta,\text{Ind}_K^G\theta)$. The associated Fourier transform is the first explicit form of \eqref{introisom3}. In \cite{CURFOS,CR2,st4} the Hecke algebra was introduced as a subalgebra of $L(G)$, then, using the theory of idempotents in group algebras,  $\text{Ind}^G_K\theta$ was identified with a subspace of $L(G)$. In Section \ref{SecHecke} of the present paper we use a different approach: the theory developed in Section \ref{SecorthFrob} naturally yields an isometric immersion of $\text{Ind}^G_K\theta$ in $L(G)$ and this isometry may be used as a tool to translate the harmonic analysis in Section \ref{Seccommutant} into a harmonic analysis of the Hecke algebra. Adapted bases for the irreducible representations of $G$ involved in the decomposition of $\text{Ind}^G_K\theta$ yield a complete set of matrix coefficients. In Section \ref{SecGelTse} we develop a theory of Gelfand-Tsetlin bases: when it is applicable, it leads to a natural orthonormal basis for $\text{Hom}_K(\theta,\text{Res}^G_K\sigma)$ and  to a corresponding basis for the $\sigma$-isotypic component of $\text{Ind}_K^G\theta$. The first one is obtained by means of iterated restrictions, while the second one  is obtained by iterated inductions.\\

It should be interesting to examine the case in which $K$ is a normal subgroup using Clifford theory (see \cite{Clifford}).  Another direction of research might be the extension of our results for permutation representations of wreath products (see \cite{tree, wreath} and \cite{st1}) to induced representations.  A parallel theory was developed by D'Angeli  and  Donno in \cite{DD1,DD2,DD3} by generalizing some constructions that arise in  the setting of {\it association schemes}; from the point of view of special functions see \cite{MIZ000}. \\

All the prerequisites for this paper may found in our books \cite{book,book2} and in our survey papers \cite{Mackey,st4}.  Motivations may be found in our preceding papers \cite{ScarabottiLapBer,st1,st2,st3}, where only permutation representations were studied and applied. Concrete examples of spherical functions associated with  induced representations are in \cite{ Mizukawa, Stembridge}, but the authors of these papers consider only multiplicity free induced representations of one-dimensional representations.\\

%%%%%%%%%%%%%%%%%%%%%

%%%%%%%%%%%%%%%%%%%%%%%%%%%%%%%%%%%%%%%%%%%%%%%%%%%%%
%%%%%%%%%%%%%%%%%%%%%%%%%%%%%%%%%%%%%%%%%%%%%%%%%%%%%
\section{Preliminaries}\label{Secprel}
%%%%%%%%%%%%%%%%%%%%%%%%%%%%%%%%%%%%%%%%%%%%%%%%%%%%%
%%%%%%%%%%%%%%%%%%%%%%%%%%%%%%%%%%%%%%%%%%%%%%%%%%%%%

In this section, in order to fix notation, we recall some basic facts on linear operators on finite dimensional spaces and on the representation theory of finite groups.
The scalar product on a finite dimensional Hermitian vector space $V$ is denoted by $\langle \cdot, \cdot\rangle_V$ and the associated norm by $\lVert\cdot \rVert_V$; we usually  omit the subscript if it is clear from the context the vector space considered. All the vector spaces will be Hermitian, and therefore we will omit this adjective.
Given two finite dimensional vector spaces $W,U$ we denote by $\Hom(W,U)$ the vector space of all linear maps  from $W$ to $U$ and for $T \in \Hom(W,U)$ we denote by $T^*$ the adjoint of $T$. We define a (normalized Hilbert-Schmidt) scalar product on  $\Hom(W,U)$ by setting 

\[
\langle T_1, T_2 \rangle_{\Hom(W,U)} = \frac{1}{\dim W}\tr(T^*_2T_1)
\]
for all $T_1, T_2 \in \Hom(W,U)$. Since $\tr(T_2^*T_1) = \tr(T_1T_2^*)$ we have

\begin{equation}\label{PS}
 \langle T_1, T_2 \rangle=\frac{\dim U}{\dim W}\langle T_2^*, T_1^* \rangle 
\end{equation}
In particular,   the map $\Hom(W,U) \ni T \mapsto\sqrt{ \frac{\dim U}{\dim W}}T^* \in \Hom(U,W)$ is a bijective isometry. Finally, note that if $I_W:W\rightarrow W$ is the identity operator then $\lVert I_W\rVert_{\Hom(W,W)}=1$.\\

We consider only {\em unitary representations} of finite groups and the adjective unitary will be usually omitted. If $\sigma$ is a representation of a finite group $G$ its dimension will be denoted by $d_\sigma$.
If $(\sigma,W)$ and $(\rho, U)$ are two representations of $G$
we denote by $\Hom_G(W,U) = \{T \in \Hom(W,U):  T\sigma(g) = \rho(g)T, \forall g \in G\}$, the space of all {\it intertwining operators}. Observe that if $T$ belongs to $\Hom_G(W,U)$  then $T^* $ belongs to $\Hom_G(U,W)$. 
Indeed, for all $g \in G$ we have

\begin{equation}\label{adjoint}
T^*\rho(g) = T^*\rho(g^{-1})^* =( \rho(g^{-1})T)^* = (T\sigma(g^{-1}))^* = \sigma(g^{-1})^*T^* = \sigma(g)T.
\end{equation}
If $(\sigma, W)$ is irreducible and $m = \dim \Hom_G(W, U)$  then $U$ contains  $m$ copies of $W$. In this case we say  that 
$T_1, T_2, \ldots, T_m \in \Hom_G(W,U)$ give rise to an {\it isometric orthogonal decomposition} of the $W$-{\it isotypic component} $mW$ of $U$ if for every $w_1, w_2 \in W$ and $i,j \in \{1,2,\ldots,m\}$  we have
\[
\langle T_iw_1, T_jw_2\rangle_U = \langle w_1, w_2\rangle_W\delta_{i,j}.
\]
This implies that the subrepresentation of $U$ isomorphic to $mW$ is equal to the orthogonal  direct sum
\[
T_1W\oplus T_2W \oplus \cdots \oplus T_mW
\]
and each operator $T_j$ is a isometry from $W$ to $\ran T_j \equiv T_jW$.

\begin{lemma}\label{lemma1}
Suppose that $(\sigma,W)$ is irreducible. Then the operators $T_1, T_2, \ldots, T_m$ give rise to an isometric  orthogonal decomposition of the $W$ component of $U$ if and only if  $T_1, T_2, \ldots, T_m$  form an orthonormal basis for $\Hom_G(W,U)$. Moreover, if this is the case, then we have:

\begin{equation}\label{orthdeltaij}
T_j^*T_i  = \delta_{i,j}I_W.
\end{equation}
\end{lemma}
\begin{proof}
Suppose that $T_1, T_2, \ldots, T_m$  form an orthonormal basis for $\Hom_G(W,U)$.
By \eqref{adjoint} we know that
  $T^*_j \in \Hom_G(U,W)$. 
Therefore  $T_j^*T_i \in \Hom_G(W,W)$ and, by Schur's lemma, there exist $\lambda_{i,j} \in \CC$ such that  $T_j^*T_i = \lambda_{i,j}I_W$. By taking the traces of both sides, we get
\[
\delta_{i,j}d_\sigma = \tr(T_j^*T_i ) =  \lambda_{i,j} d_\sigma \Rightarrow \lambda_{i,j} = \delta_{i,j},
\]
that is \eqref{orthdeltaij}. Therefore, if $w_1, w_2 \in W$  then 
\[
\langle T_i w_1, T_jw_2 \rangle_U = \langle T_j^* T_i w_1, w_2 \rangle_W = \delta_{i,j}\langle w_1, w_2 \rangle_W.
\]
The converse  implication is trivial.
\end{proof}

Let $(\sigma,W)$ be a representation of $G$ and let $\{w_1,w_2, \ldots, w_{d_\sigma}\}$ an orthonormal basis of $W$.  The corresponding matrix coefficients are defined by setting 
\begin{equation}\label{stellap2}
u_{j,i}^\sigma(g) = \langle \sigma(g)w_i, w_j\rangle
\end{equation}
for $i,j = 1,2, \ldots, d_\sigma$ and $g \in G$.
If $\sigma, \rho$ are irreducible $G$-representations we set $\delta_{\sigma,\rho}= 1$, if $\sigma$ and $\rho$ are equivalent, otherwise we set $\delta_{\sigma,\rho}= 0$.

Let $L(G) = \{f: G \to \CC\}$ be the vector space of all complex valued functions defined on  $G$, endowed with the scalar product  
$\langle f_1, f_2\rangle = \sum_{g \in G}f_1(g)\overline{f_2(g)}$  for  $f_1, f_2 \in L(G)$.   It  also has  a natural  structure of algebra by defining the convolution product of   $f_1,f_2 \in L(G)$ as the function
$(f_1*f_2)(g) = \sum_{g_0\in G}f_1(gg_0^{-1})f_2(g_0)$, for all $g\in G$ .

\begin{proposition}
Let $(\sigma,W)$ and $(\rho,U)$ be irreducible $G$-representations. Then 
\begin {enumerate}
\item{
\begin{equation}\label{ORT}
 \langle u_{i,j}^\sigma, u_{h,k}^\rho\rangle= \frac{|G|}{d_\sigma} \delta_{\sigma, \rho}\delta_{i,h}\delta_{j,k} 
\mbox{   (orthogonality relations)}
\end{equation}
}
\item{
\begin{equation}\label{CON} 
u_{i,j}^\sigma* u_{h,k}^\rho = \frac{|G|}{d_\sigma} \delta_{\sigma, \rho}\delta_{j,h} u^\sigma_{i,k}
\mbox{    (convolution property).}
\end{equation}}
\end{enumerate}
\end{proposition}
\begin{proof}
See  Lemma 3.6.3 and Lemma 3.9.14  of \cite{book}.
\end{proof}

 Let $(\sigma,W)$ be a $G$-representation and   denote by $\chi^\sigma$ its character. The following elementary formula is a generalization of (2) in Exercise 9.5.8 of \cite{book}.

\begin{proposition}
If  $(\sigma, W)$ is irreducible, $w \in W$ is a vector of norm 1 and $\phi(g) = \langle\sigma(g) w, w\rangle$ is the diagonal  matrix
coefficient associated with $w$, then
\begin{equation}\label{eD21}
\chi^\sigma(g) = \frac{d_\sigma}{|G|}\sum_{h \in G}\phi(h^{-1}gh)
\end{equation}
for all $g\in G$.
\end{proposition}
\begin{proof}
Let $\{w_1 = w, w_2, \ldots, w_{d_\sigma}\}$ be an orthonormal basis of $W$ and
$u^\sigma_{j,i}$ as in \eqref{stellap2}; then 
\[
\sigma(g)w_i = \sum_{j = 1}^{d_\sigma} u^\sigma_{j,i}(g)w_j.
\]
 Thus 
\[
\begin{split}
\sum_{h \in G}\phi(h^{-1}gh) &= \sum_{h \in G}\langle\sigma(g)\sigma(h) w_1, \sigma(h) w_1\rangle\\
& = \sum_{j, \ell = 1}^{d_\sigma}\sum_{h \in G}u^\sigma_{j,1}(h)\overline{u^\sigma_{\ell,1}(h)}\langle\sigma(g) w_j,  w_\ell\rangle\\
\mbox{(by \eqref{ORT})}\ \ \ \ & = \frac{|G|}{d_\sigma}\chi^\sigma(g).
\end{split}
\]
\end{proof}

Let $K$ be a subgroup of $G$,   $(\theta, V)$ a representation of $K$ and denote by $\lambda = \Ind_K^G\theta$ the {\it induced representation} (see for instance,\cite{Bump,Mackey,book2,NS} and \cite{Sternberg}). We recall that 
the representation space is given by
\begin{equation}\label{H31}
\Ind_K^GV = \{f:G  \to V: f(gk) = \theta(k^{-1})f(g), \  \forall g \in G, k \in K\}
\end{equation} and that the $G$-action is defined by setting
\begin{equation}\label{H32}
[\lambda(g_0) f](g) = f(g_0^{-1}g)
\end{equation}
for all $f \in \Ind_K^GV$, $g,g_0 \in G$.
Let $G  = \coprod_{t \in \TT}tK$ be a decomposition of $G$ into right $K$-cosets ($\coprod$ denotes a disjoint union).
For $v \in V$ we define  $f_v \in \Ind_K^GV$  by setting 
\begin{equation}\label{definizione stellata1}
f_v(g) = \left\{
\begin{array}{ll}
\theta(g^{-1})v & \mbox{if $g \in K$}\\
0 & \ \mbox{if $g \notin K$.}
\end{array}
\right.
\end{equation}
Then for every  $f \in \Ind_K^GV$ we have:

\begin{equation}\label{definizione stellata2}
f = \sum_{t\in \TT}\lambda(t)f_{v_t}
\end{equation}
with $v_t = f(t)$. 
The representation $\Ind_K^G\theta$ is unitary with respect to the following scalar product:
\begin{equation}\label{H33}
\langle f_1, f_2\rangle_{\Ind_K^GV}  = \frac{1}{|K|} \sum_{g \in G}\langle f_1(g), f_2(g) \rangle_V.
\end{equation}
Moreover, if $\{v_j:j=1,2,\dotsc,d_\theta\}$ is an orthonormal basis in $V$ then the set

\begin{equation}\label{orthbasisind}
\{\lambda(t)f_{v_j}:t\in\TT, j=1,2,\dotsc,d_\theta\}
\end{equation}
is an orthonormal basis in $\text{Ind}_K^G V$ (see \cite{Mackey}).

%%%%%%%%%%%%%%%%%%%%%%%%%%%%%%%%%%%%%%%%%%%%%%%%%%%%%
%%%%%%%%%%%%%%%%%%%%%%%%%%%%%%%%%%%%%%%%%%%%%%%%%%%%%
\section{Orthogonality relations for Frobenius reciprocity}\label{SecorthFrob}
%%%%%%%%%%%%%%%%%%%%%%%%%%%%%%%%%%%%%%%%%%%%%%%%%%%%%
%%%%%%%%%%%%%%%%%%%%%%%%%%%%%%%%%%%%%%%%%%%%%%%%%%%%%

Let $G$ be again a finite group, $K\leq G$ a subgroup, $(\sigma,W)$ a representation of $G$ and $(\theta, V)$ a  representations of $K$. Frobenius reciprocity is usually stated an explicit isomorphism between $\Hom_G(W,\Ind_K^GV)$ and $\Hom_K(\Res_K^GW, V)$.
In this section we present a detailed analysis of all {\it other} aspects of Frobenius reciprocity, all of them in an {\em orthogonal} version. In particular, we show  how to obtain from an explicit orthogonal decomposition of the $V$-isotypic component of $\Res^G_KW$, an explicit orthogonal  decomposition of the 
$W$-isotypic component of $\Ind_K^G V$. 
  
\begin{definition}
{\rm
\begin{enumerate}
\item{
For each $T \in \Hom_G(W, \Ind_K^GV)$ we set 
\[
\TW w = \sqrt{|G/K|}[Tw](1_G), \qquad\mbox{ for all $w \in W$.}
\]
}
\item{ For each $L \in \Hom_K(\Res_K^GW, V)$ we set 
\[
[\LV w](g)= \frac{1}{\sqrt{|G/K|}}L\sigma(g^{-1})w,\qquad \mbox{ for all $w \in W$, $g\in G$.}
\]
}
\item{For each $T \in \Hom_G(\Ind_K^G V, W)$ we set
\[
\overset{\square}{T} v =\sqrt{ |G/K|} Tf_v,\qquad \mbox{ for all $v \in V$.}
\]
}
\item{For  each $L \in \Hom_K(V, \Res_K^GW)$ we set 
\[
\overset{\diamond}{L}f = \frac{1}{\sqrt{|G/K|}}\sum_{t \in \TT}\sigma(t)Lf(t),\qquad \mbox{ for all $f \in \Ind_K^GV$.}
\]
}
\end{enumerate}}
\end{definition}

\noindent
Note that 
\begin{equation}\label{definizione stellata3}
\overset{\diamond}{L} f   = \frac{1}{\sqrt{|G|\cdot |K|}}\sum_{g \in G}\sigma(g)Lf(g).
\end{equation}
Indeed, 
\[
\sum_{g \in G}\sigma(g)Lf(g) = \sum_{t \in \TT}\sum_{k \in K}\sigma(tk)Lf(tk) = 
|K|\sum_{t \in \TT}\sigma(t)Lf(t)
\]
because $L \in \Hom_K(V, \Res^G_KW)$ and $\theta(k)f(tk) = f(t)$. In particular, $\overset{\diamond}{L}$ does not depend on the particular choice of $\TT$.\\

\begin{theorem}[Frobenius reciprocity revisited]\label{FRR}
\begin{enumerate}
\item{For each $T \in \Hom_G(W, \Ind_K^GV)$ we have $\TW\in \Hom_K(\Res^G_KW, V)$ and the map
\[
\begin{array}{ccc}
\Hom_G(W, \Ind_K^GV) & \longrightarrow & \Hom_K(\Res_K^G W, V)\\
T & \longmapsto & \TW
\end{array}
\]
is a linear isometric isomorphism.  Moreover, its inverse is given by
\[
\begin{array}{ccc}
 \Hom_K(\Res_K^G W, V)  & \longrightarrow & \Hom_G(W, \Ind_K^GV)\\
 L& \longmapsto & \LV .
\end{array}
 \]}
 \item{For each $T\in \Hom_G(W, \Ind_K^GV)$ we have:
 \begin{equation}\label{commTwdgecircle}
(T^*)^\square = (\TW)^*.
\end{equation}
 }
\end{enumerate}
\end{theorem}
\begin{proof}
(1) Let $T \in \Hom_G(W, \Ind_K^GV)$ and $\lambda$ as in \eqref{H32}. For all $k \in K$ and $w\in W$ we have 
\[
\begin{split}
\TW\sigma(k)w  & = \sqrt{|G/K|} [T\sigma(k)w](1_G)\\
\mbox{(since $T \in \Hom_G(W, \Ind_K^GV)$)} \  \  \  \   \  \ & = \sqrt{|G/K|}[\lambda(k)Tw](1_G)\\
\mbox{(by \eqref{H32})}\  \  \  \  \  \  & = \sqrt{|G/K|} [Tw](k^{-1})\\
\mbox{(by \eqref{H31})}\  \  \  \  \  \  & =  \sqrt{|G/K|}\theta(k)[Tw](1_G)\\
& =  \theta(k)\TW w.
\end{split}
\]
This proves that $\TW \in \Hom_K(\Res_K^GW,V)$.
The identity
\begin{equation}\label{H6stella}
\begin{split}
[Tw](g)& = [\lambda(g^{-1})Tw](1_G)\\
& = [T\sigma(g^{-1})w](1_G)\\
& = \frac{1}{\sqrt{|G/K|}}\TW \sigma(g^{-1})w
\end{split}
\end{equation}
shows that the map $T \mapsto \TW$ is injective, because $T$ is determined by $\TW$.
Now we use \eqref{H6stella} to show that the map is also an isometry. 
If $T_1, T_2 \in \Hom_G(W, \Ind_K^GV)$ and  $\{w_1, w_2, \ldots, w_{d_\sigma}\}$ is an orthonormal basis of $W$ then
\[
\begin{split}
\tr(T_2^*T_1) & = \sum_{i = 1}^{d_\sigma}\langle T_1w_i, T_2w_i\rangle_{\Ind_K^GV}\\
\mbox{(by \eqref{H33}) } \  \  \  \  \  \ & =  \sum_{i = 1}^{d_\sigma}\frac{1}{|K|}\sum_{g \in G}\langle[T_1w_i](g), [T_2w_i](g)\rangle_V\\
\mbox{(by \eqref{H6stella})} \  \  \  \  \  \  & = \sum_{i = 1}^{d_\sigma}\frac{1}{|G|}\sum_{g \in G}
\langle \overset{\wedge}{T_1}\sigma(g^{-1})w_i, \overset{\wedge}{T_2}\sigma(g^{-1})w_i\rangle_V\\
& = \frac{1}{|G|}\sum_{g\in G}\tr[\sigma(g)(\overset{\wedge}{T_2})^*\overset{\wedge}{T_1}\sigma(g^{-1})]\\
&= \tr[(\overset{\wedge}{T_2})^*\overset{\wedge}{T_1}],
\end{split}
\]
that is $\langle T_1,T_2\rangle =\frac{1}{d_\sigma}{\rm tr}(T_2^*T_1)=\frac{1}{d_\sigma}{\rm tr \left[(\widehat{T}_2)^*\widehat{T}_1\right]}=\langle \widehat{T}_2,\widehat{T}_1\rangle$.
It is easy to see that if $L\in \Hom_K(\Res_K^GW,V)$ then 
$[\LV w](gk) = \theta(k^{-1})[\LV w](g)$ and  $\lambda(g)\LV w = \LV\sigma(g) w$ for all $g\in G,\  k \in K,\ w\in W$, that is, $\LV w \in \Ind_K^GV$ and $\LV \in \Hom_G(W, \Ind_K^GV)$.
Finally, by definition of $\wedge$ and $\vee$ we have 
\[
\Bigl(\LV\Bigr)^{\wedge} w =\sqrt{ |G/K}|[\LV w](1_G) = L w
\]
for all $w \in W$, that is  the map $T\mapsto \TW$ is surjective and $L\mapsto \LV$ is its inverse.

(2) For any $T\in \Hom_G(W,\Ind_K^GV)$, $w\in W$ and $v\in V$ we have (by definition of $\square$):
\[
\begin{split}
\frac{1}{\sqrt{|G/K|}}\langle (T^*)^\square v, w\rangle_W & = \langle T^*f_v, w\rangle_W\\
& = \langle f_v, Tw\rangle_{\Ind_K^GV}\\
& = \frac{1}{|K|}\sum_{g\in G}\langle f_v(g), [Tw](g)\rangle_V\\
\mbox{(by \eqref{definizione stellata1})} \  \  \  \  \  \   & = \frac{1}{|K|}\sum_{k \in K}\langle\theta(k^{-1})v, [Tw](k)\rangle_V\\
& = \frac{1}{|K|}\sum_{k \in K}\langle v, \theta(k)[Tw](k)\rangle_V\\
\mbox{ (since $Tw\in \Ind_K^GV$)} \  \  \  \  \  \  & =\frac{1}{\sqrt{|G/K|}} \langle v, \TW w\rangle_V\\
& =\frac{1}{\sqrt{|G/K|}} \langle(\TW)^*v, w\rangle_W.
\end{split}
\]
\end{proof}

The following corollary should be compared with Corollary 34.1 in \cite{Bump}  and Section 2.3 in \cite{st1}.\\

\begin{corollary}[The other side of Frobenius reciprocity]\label{corotherside}
Let $T\in \Hom_G(\Ind_K^G V, W)$. Then  $\overset{\square}{T}\in \Hom_K(V, \Res_K^G W)$ and the map
\[
\begin{array}{ccc}
\Hom_G(\Ind_K^GV, W) & \longrightarrow & \Hom_K(V, \Res_K^GW)\\
T & \longmapsto & \overset{\square}{T}
\end{array}
\]
is a linear isomorphism with 

\begin{equation}\label{normcircle}
\langle \overset{\square}{T}_1,\overset{\square}{T}_2\rangle=\lvert G/K\rvert \langle T_1,T_2\rangle,
\end{equation}
for all $T_1,T_2\in\Hom_G(\Ind_K^G V, W)$. The inverse is given by 
\[
\begin{array}{ccc}
\Hom_K(V, \Res_K^GW)& \longrightarrow &  \Hom_G(\Ind_K^GV, W)\\
  L &  \longmapsto & \overset{\diamond}{L}
\end{array}\]
and 
\begin{equation}\label{pastrocchio}
(L^*)^\vee = \left(\overset{\diamond}{L}\right)^\square
\end{equation}
for all $L\in \Hom_K(V, \Res_K^GW)$.

In particular, the diagram
 \[
\begin{array}{ccc}
\Hom_G(W,\Ind_K^G V)&\stackrel{\wedge}{\longrightarrow}
& \Hom_K(\Res_K^GW,V)\\
*\downarrow &&*\downarrow \\
\Hom_G(\Ind_K^G V,W)&\stackrel{\square}{\longrightarrow}
& \Hom_K(V,\Res_K^GW)\\
\end{array}
\]
is commutative.
\end{corollary}
\begin{proof}
Besides the  statement that the map $L\mapsto \overset{\diamond}{L}$ is the inverse of the map $T \mapsto \overset{\square}{T}$, everything follows from (2) in Theorem \ref{FRR}, \eqref{PS} and \eqref{adjoint}.
For  all $T \in \Hom_G(\Ind_K^GV, W)$, $\phi\in \Ind_K^GV$, we have:

\begin{align*}
(\overset{\square}T)^\diamond \phi & = \frac{1}{\sqrt{|G/K|}}\sum_{t\in \TT}\sigma(t)\overset{\square}T\phi(t)&\mbox{ (by definition of $\diamond$)}  \\
 & = \sum_{t\in \TT}\sigma(t)Tf_{v_t}&\mbox{(by definition of $\square$ with $v_t= \phi(t)$)} \\
 & = \sum_{t\in \TT}T\lambda(t)f_{v_t}&\mbox{(since $T \in \Hom_G(\Ind_K^GV, W)$)}  \\
 & = T\phi &\mbox{(by \eqref{definizione stellata2})}
\end{align*}
that is, the map $L\mapsto \overset{\diamond}L$ is the inverse of the map 
$T\mapsto \overset{\square}T$. We want also show how to derive \eqref{normcircle}. For $T_1,T_2\in\Hom_G(\Ind_K^G V, W)$ we have:

\begin{align*}
\langle \overset{\square}{T}_1,\overset{\square}{T}_2\rangle=&\langle[(T_1^*)^\wedge]^*,[(T_2^*)^\wedge]^*\rangle &(\text{by }\eqref{commTwdgecircle})\\ 
=&\frac{d_\sigma}{d_\theta}\langle (T_2^*)^\wedge,(T_1^*)^\wedge\rangle &(\text{by }\eqref{PS})\\
=&\frac{d_\sigma}{d_\theta}\langle T_2^*,T_1^*\rangle &(\text{by Theorem }\ref{FRR})\\
=&\lvert G/K\rvert \langle T_1,T_2\rangle. &(\text{again by }\eqref{PS})
\end{align*}

Finally, by \eqref{commTwdgecircle} we have 
\begin{equation}\label{mostro}
\left\{\left[(L^*)^\vee\right]^*\right\}^\square =  \left\{\left[(L^*)^\vee\right]^\wedge\right\}^* = L = 
 \left\{\left[\left(\overset{\diamond}{L}\right)^*\right]^*\right\}^\square
\end{equation}
and this yields \eqref{pastrocchio}.
\end{proof}

\begin{corollary}[Orthogonality relations for Frobenius reciprocity I]\label{corollario1}

Let $m $ be the dimension of $\Hom_G(W, \Ind_K^GV)$ and suppose that $L_1, L_2, \ldots, L_m\in \Hom_K(V, \Res_K^GW)$. Then the following facts are equivalent:
\begin{enumerate}

\item{The set $\{L_1, L_2, \ldots, L_m\}$ is an orthonormal basis of $\Hom_K(V, \Res_K^GW)$;}

\item{The set $\left\{\sqrt{|G/K|}\overset{\diamond}{L_1},\sqrt{|G/K|} \overset{\diamond}{L_2}, \ldots, \sqrt{|G/K|}\overset{\diamond}{L_m}\right\}$
is an orthonormal basis of 

$\Hom_G(\Ind_K^GV, W)$;}

\item{The set $\left\{\sqrt{\frac{d_\sigma}{d_\theta}}L_1^*, \sqrt{\frac{d_\sigma}{d_\theta}} L_2^*, \ldots, \sqrt{\frac{d_\sigma}{d_\theta}} L_m^*\right\}$
is an orthonormal basis of 

$\Hom_K(\Res_K^G W, V)$;}

\item{The set $\left\{\sqrt{\frac{d_\sigma}{d_\theta}}(L_1^*)^\vee,\sqrt{\frac{d_\sigma}{d_\theta}} (L_2^*)^\vee, \ldots,\sqrt{\frac{d_\sigma}{d_\theta}} (L_m^*)^\vee\right\}$ 
is an orthonormal basis  of  

$\Hom_G(W, \Ind_K^GV)$.}

\end{enumerate}
\end{corollary}

\quad\\

\begin{corollary}[Orthogonality relations for Frobenius reciprocity II]\label{corollario2}
Suppose that $(\sigma, W)$ and $ (\theta,V)$ are irreducible.  Then 
$\sqrt{\frac{d_\sigma}{d_\theta}}(L_1^*)^\vee,\sqrt{\frac{d_\sigma}{d_\theta}} (L_2^*)^\vee, \ldots,\sqrt{\frac{d_\sigma}{d_\theta}} (L_m^*)^\vee$ give rise to an isometric orthogonal decomposition of the $W$-component of $\Ind_K^GV$ if and only if  $L_1, L_2, \ldots, L_m$ give rise to an isometric orthogonal decomposition of the $V$-component of $\Res_K^G W$.
\end{corollary}
\begin{proof}
It follows from Theorem \ref{FRR} and Lemma \ref{lemma1}.
\end{proof}

In the applications of the last two Corollaries, we will often use the identity \eqref{pastrocchio}, that is  
$(L_j^*)^\vee =(\overset{\diamond}{L_j})^*$.
The following commutative diagram is helpful to memorize the previous results.
\[
\begin{array}{ccc}
\text{Hom}_G(W,\text{Ind}_K^G V)&\begin{array}{c}\stackrel{\wedge}{\longrightarrow}\\
\stackrel{\vee}{\longleftarrow}
\end{array}& \text{Hom}_K(\text{Res}_K^GW,V)\\
*\updownarrow &&\updownarrow*\\
\text{Hom}_G(\text{Ind}_K^G V,W)&\begin{array}{c}\stackrel{\square}{\longrightarrow}\\
\stackrel{\lozenge}{\longleftarrow}
\end{array}& \text{Hom}_K(V,\text{Res}_K^GW)\\
\end{array}
\]

\begin{remark}{\rm
In \cite{Nebe} it is developed a different version of orthogonality relations for Frobenius reciprocity. Actually, the author works in a more general setting: she considers representations over fields of characteristic zero and her spaces are endowed with arbitrary non-degenerate symmetric bilinear forms. However, we limit ourselves to illustrate and derive her main result in our setting. Theorem 2.1 of \cite{Nebe} may be expressed in the following way: under the assumption that $W$ is $G$-irreducible, if $L\in \Hom_K(V,\Res_K^GW)$ is an isometry then also $\sqrt{\frac{d_\sigma}{d_\theta}}(L^*)^\vee \equiv\sqrt{\frac{d_\sigma}{d_\theta}}(\overset{\diamond}{L})^*\in \Hom_G(W,\Ind_K^G V)$ is an isometry. This is our derivation: if $L$ is an isometry, then $\lVert L\rVert=1$ and therefore also $\lVert \sqrt{\frac{d_\sigma}{d_\theta}}(L^*)^\vee\rVert=1$. Arguing as in Lemma \ref{lemma1}, it is easy to show that this fact implies that $\sqrt{\frac{d_\sigma}{d_\theta}}(L^*)^\vee$ is an isometry. Finally, we note that Theorem 2.4 in \cite{Nebe} is a version of our Corollary \ref{corollario2}.
}
\end{remark}

\begin{remark}{\rm
Corollary \ref{corotherside} is useful when irreducible representations are obtained as induced representations. This is the case of the little group method of Mackey and Wigner: we refer to \cite{Clifford} for a general formulation of this method and to \cite{wreath} for its applications to wreath products of finite groups. Indeed, in \cite{st1} we used a version of Corollary \ref{corotherside} in order to decompose a wide class of permutation representations of wreath products, including the exponentiation action.
}
\end{remark}

%%%%%%%%%%%%%%%%%%%%%%%%%%%%%%%%%%%%%%%%%%%%%%%%%%%%%
%%%%%%%%%%%%%%%%%%%%%%%%%%%%%%%%%%%%%%%%%%%%%%%%%%%%%
\section{Harmonic analysis in $\Hom_G(\Ind_K^GV, \Ind_K^GV)$}\label{Seccommutant}
%%%%%%%%%%%%%%%%%%%%%%%%%%%%%%%%%%%%%%%%%%%%%%%%%%%%%
%%%%%%%%%%%%%%%%%%%%%%%%%%%%%%%%%%%%%%%%%%%%%%%%%%%%%

In thi section we construct  an orthonormal basis of the commutant of  $\Ind_K^G V$  from the orthonormal bases  analyzed in the 
previous section.  This  way we can introduce a Fourier transform that gives an explicit isomorphism between 
 $\Hom_G(\Ind_K^GV, \Ind_K^GV)$ and   $\bigoplus_{\sigma\in J}M_{m_\sigma,m_\sigma}(\mathbb{C})$.

Let $(\theta,V)$ be an irreducible representation of $K \leq G$  and  $(\sigma, W)$ an irreducible representation of $G$.  Consider $L_1, L_2 \in \Hom_K(V, \Res_K^GW)$. 
Then  we have  $L_1^* \in  \Hom_K(\Res_K^GW, V)$ and 
\[
\Ind_K^G V \stackrel{\overset{\diamond}{L_2}}{\longrightarrow} W  \stackrel{(L_1^*)^\vee}{\longrightarrow} \Ind_K^GV,
\]
that is  ($\overset{\diamond}{L_1})^* \overset{\diamond}{L_2} = (L_1^*)^\vee \overset{\diamond}{L_2} \in 
\Hom_G(\Ind_K^GV, \Ind_K^GV)$.

\begin{lemma}\label{lemma3}
Let   $(\sigma_1, W_1)$ and  $(\sigma_2, W_2)$  be two irreducible inequivalent  representations  of $G$. Consider 
  $L_1, L_2 \in \Hom_K(V, \Res_K^G W_2)$ and  $L_3, L_4 \in \Hom_K(V, \Res_K^GW_1)$. Then 
  
\begin{equation}\label{L3L1}
\overset{\diamond}{L_3}( \overset{\diamond}{L_1})^*=0
\end{equation}
and
  
\[
\left\langle ( \overset{\diamond}{L_1})^* \overset{\diamond}{L_2},  (\overset{\diamond}{L_3})^* \overset{\diamond}{L_4}\right \rangle = 0.
\]
\end{lemma}
\begin{proof}
By Schur's lemma, $\overset{\diamond}{L_3}( \overset{\diamond}{L_1})^*  \in  \Hom_G(W_2, W_1) = \{0\}$. Moreover, by definition of scalar product in  $ \Hom_G(\Ind_K^GV, \Ind_K^G V)$ we have 
\[
\left\langle ( \overset{\diamond}{L_1})^* \overset{\diamond}{L_2},  (\overset{\diamond}{L_3})^* \overset{\diamond}{L_4}\right\rangle 
=\frac{1}{\dim \Ind_K^GV}\tr\left[( \overset{\diamond}{L_4})^* \overset{\diamond}{L_3}( \overset{\diamond}{L_1})^* \overset{\diamond}{L_2}\right]=0.
\]
\end{proof}

\begin{lemma}\label{lemma4}
Let $(\sigma, W)$ be an irreducible representation of $G$  and $\{L_1, L_2, \ldots, L_m\}$ an orthonormal basis of $\Hom_K(V, \Res_K^G W)$.
Then

\begin{equation}\label{LdiamLdiamst}
\overset{\diamond}{L_h}(\overset{\diamond}{L_i})^*  =\frac{d_\theta}{d_\sigma} I_W\delta_{i,h}.
\end{equation}
and the operators $ (\overset{\diamond}{L_i})^* \overset{\diamond}{L_j} \in \Hom_G(\Ind_K^GV, \Ind_K^G V)$  satisfy the orthogonality relations:
\[
\left\langle  (\overset{\diamond}{L_i})^* \overset{\diamond}{L_j},  (\overset{\diamond}{L_h})^* \overset{\diamond}{L_\ell}\right\rangle = \delta_{i,h}\delta_{j,\ell}\frac{d_\theta}{d_\sigma|G/K|}.
\]
\end{lemma}
\begin{proof}
The identity \eqref{LdiamLdiamst} follows from \eqref{orthdeltaij} and Corollary \ref{corollario2}.
Moreover,

\[
\begin{split}
\left\langle (\overset{\diamond}{L_i})^* \overset{\diamond}{L_j},  (\overset{\diamond}{L_h})^* \overset{\diamond}{L_\ell}\right\rangle &  = \frac{1}{\dim \Ind_K^G V}\tr\left[( \overset{\diamond}{L_\ell})^* \overset{\diamond}{L_h}
( \overset{\diamond}{L_i})^* \overset{\diamond}{L_j}\right]\\
(\text{by }\eqref{LdiamLdiamst})\qquad&=\frac{ \delta_{i,h}d_\theta}{d_\sigma \dim \Ind_K^G V}\tr\left[( \overset{\diamond}{L_\ell})^* \overset{\diamond}{L_j}\right]\\
  & = \delta_{i,h}\frac{d_\theta}{d_\sigma}\langle \overset{\diamond}{L}_j, \overset{\diamond}{L}_\ell\rangle\\
\mbox{ (by 2. in Corollary \ref{corollario1})}\ \ \ \ & =  \delta_{i,h}\delta_{j, \ell} \frac{d_\theta}{d_\sigma|G/K|}.
\end{split}
\]
\end{proof}

In what follows, we denote by $M_{m,m}(\CC)$ the algebra of all $m\times m$ complex matrices.
Let $\Ind_K^GV = \bigoplus_{\sigma \in J}m_\sigma W_\sigma$ be the decomposition  of $\Ind_K^GV$ into irreducible representations of $G$ (i.e., $\{(\sigma, W_\sigma):\sigma\in J\}$ is  a complete  set of all irreducible inequivalent  representations of $G$ contained in $\Ind_K^GV$ and $m_\sigma$ is the multiplicity of $W_\sigma$  in $\Ind_K^GV$).
For every $\sigma \in J$ select an orthonormal basis

\begin{equation}\label{choicebasis} 
\{L_{\sigma, 1}, L_{\sigma, 2}, \ldots, L_{\sigma, m_\sigma}\}
\end{equation}
of $\Hom_K(V, \Res_K^GW_\sigma)$ and set

\begin{equation}\label{quad11}
U^\sigma_{i,j}=\frac{d_\sigma}{d_\theta}(\overset{\diamond}L_{\sigma, i})^* \overset{\diamond}L_{\sigma, j},
\end{equation}
for $i,j=1,2,\dotsc,m_\sigma$. For every $T\in \Hom_G(\Ind_K^G V, \Ind_K^GV)$ and $\sigma\in J$, the {\em Fourier transform} of $T$ at $\sigma$ associated to the choice of \eqref{choicebasis} is the following matrix in $M_{m_\sigma,m_\sigma}(\CC)$:

\[
[\mathcal{F}T(\sigma)]_{i,j}=\frac{d_\theta|G/K|}{d_\sigma}\langle T,U^\sigma_{i,j} \rangle,\qquad\quad i,j=1,2,\dotsc,m_\sigma.
\]
In the following theorem we will show that the Fourier transform is an explicit form of the isomorphism

\begin{equation}\label{isomorphism}
\Hom_G(\Ind_K^G V, \Ind_K^GV)  = \bigoplus_{\sigma \in J}\Hom_G(m_\sigma W_\sigma, m_\sigma W_\sigma)\cong \bigoplus_{\sigma \in J}M_{m_\sigma,m_\sigma}(\CC).
\end{equation}
We need further notation. Every element in the algebra $\bigoplus_{\sigma \in J}M_{m_\sigma,m_\sigma}(\CC)$ may be represented in the form $\bigoplus_{\sigma \in J}A_\sigma$, where $A_\sigma\in M_{m_\sigma,m_\sigma}(\CC)$. In particular,
 given  $T $ in $\Hom_G(\Ind_K^GV, \Ind_K^GV)$, we set 
\begin{equation}\label{2quad11}
\mathcal{F}T = \bigoplus_{\sigma \in J}\mathcal{F}T(\sigma).
\end{equation}
We recall \cite{book2} that the irreducible representations of this algebra are given by the natural action of each $M_{m_\sigma,m_\sigma}(\CC)$ on $\CC^{m_\sigma}$ and that \cite{st4} the corresponding irreducible characters are the functions $\{\varphi^\sigma:\sigma\in J\}$ given by: $\varphi^\rho\left(\bigoplus_{\sigma \in J}A_\sigma\right)=\text{tr}(A_\rho)$. Under the isomorphism \eqref{isomorphism}, the  irreducible representation of $\Hom_G(\Ind_K^G V, \Ind_K^GV)$ corresponding to $\sigma \in J$ is given by its action on the space $\Hom_G(W_\sigma,\Ind_K^GV)$, that is by the map $S\mapsto TS$, where $T\in \Hom_G(\Ind_K^G V, \Ind_K^GV)$ and $S\in \Hom_G(W_\sigma V, \Ind_K^GV)$.  In what follows, we will indicate by $\varphi^\sigma$ also the character of the isomorphic algebra $\Hom_G(\Ind_K^G V, \Ind_K^GV)$.

\begin{theorem}\label{teorema5}
\begin{enumerate}
\item{
The set 
\begin{equation}\label{stellaH14}
\left\{\sqrt{\frac{d_\theta|G/K|}{d_\sigma}}U^\sigma_{i,j}: \sigma \in  J, i,j = 1, 2  \ldots, m_\sigma\right\}
\end{equation}
is an orthonormal  basis  of  $\Hom_G(\Ind_K^GV, \Ind_K^G V)$. In particular, the {\em Fourier inversion formula} is:

\[
T=\sum_{\sigma\in J}\sum_{i,j=1}^{m_\sigma}[\mathcal{F}T(\sigma)]_{i,j}U^\sigma_{i,j}.
\]

}

\item{ Setting $T_{\sigma,i}=\sqrt{\frac{d_\sigma}{d_\theta}}(\overset{\diamond}L_{\sigma, i})^*$, we have  the isometric orthogonal decomposition
\[
\Ind_K^GV = \bigoplus_{\sigma \in J}\bigoplus_{i = 1}^{m_\sigma}T_{\sigma, i} W_\sigma,
\]
and the corresponding explicit isomorphism

\[
\begin{array}{ccc}
\Hom_G(\Ind_K^G V, \Ind_K^GV)  & \longrightarrow & \bigoplus_{\sigma \in J}M_{m_\sigma,m_\sigma}(\CC)\\
&&\\
 T& \longmapsto & \mathcal{F}T.
\end{array}
\]}

\item{ The operator $U^\sigma_{i,j}$ intertwines the subspace 
$T_{\sigma, j} W_\sigma$  with   $T_{\sigma, i} W_\sigma$.}

\item{The operator $U^\sigma_{i,i}$ is the orthogonal projection of $\Ind_K^GV$ onto $T_{\sigma, i} W_\sigma$.}

\item{The irreducible characters of $\Hom_G(\Ind_K^G V, \Ind_K^GV)$ are the functions $\{\varphi^\sigma: \sigma\in J\}$ given by:

\[
\varphi^\sigma(T)=\text{\rm tr}\left[\mathcal{F}T(\sigma)\right],
\]
for every $T\in\Hom_G(\Ind_K^G V, \Ind_K^GV)$.
 }
\end{enumerate}
\end{theorem}
\begin{proof} From Lemma \ref{lemma3} and Lemma \ref{lemma4} we deduce that the set \eqref{stellaH14} is orthonormal. Moreover, $\dim \Hom_G(\Ind_K^GV, \Ind_K^G V) = \sum_{\sigma \in J}m_\sigma^2$ so that  it is a basis.  For the other assertions just note that $T_{\sigma, i}$ is an isometry and that, by \eqref{L3L1} and \eqref{LdiamLdiamst}, we have

\[
U^\sigma_{i,j}U^\rho_{h,l}=\delta_{\sigma,\rho}\delta_{j,h}U^\sigma_{i,l},
\]

\[
U^\sigma_{i,j}T_{\rho,h}=\delta_{\sigma,\rho}\delta_{j,h}T_{\sigma, i},
\]
and therefore 

\[
TT_{\sigma, j}=\sum_{i=1}^{m_\sigma}[\mathcal{F}T(\sigma)]_{i,j}T_{\sigma, i},
\]
for all $\sigma,\rho\in J$, $i,j=1,2,\dotsc, m_\sigma$, $h,l=1,2,\dotsc,m_\rho$ and $T\in\Hom_G(\Ind_K^GV,\Ind_K^GV)$.

\end{proof}

%%%%%%%%%%%%%%%%%%%%%%%%%%%%%%%%%%%%%%%%%%%%%%%%%%%%%
%%%%%%%%%%%%%%%%%%%%%%%%%%%%%%%%%%%%%%%%%%%%%%%%%%%%%
\section{Harmonic analysis in the Hecke algebra}\label{SecHecke}

Let $G$ be a finite group and $K \leq G$ a subgroup. As in the previous sections we denote  by $(\theta,V)$ an irreducible $K$-representation   and by $(\sigma,W)$ an irreducible  $G$-representation.
 The left regular representation of 
$G$ on $L(G)$ is denoted by   $\lambda_G$  (to distinguish it  from $\lambda =\Ind_K^G \theta$); it  is defined by setting
$[\lambda_G(g)f](g_0) = f(g^{-1}g_0)$ for all $f \in L(G)$, $g, g_0 \in G$.

We choose $v \in V$ with $\|v\| = 1$ and define $\psi \in L(K)$  by setting
\[
\psi(k) = \frac{d_\theta}{|K|}\langle v , \theta(k) v\rangle, \  \  \  \forall k \in K.
\]
Since $L(K) \subseteq L(G)$, we may consider $\psi$ also as a function in $L(G)$. We define the operator 
\[
T_v: \Ind_K^GV \longrightarrow  L(G)
\]
by setting:
\[
(T_vf)(g) = \sqrt{d_\theta/|K|}\langle f(g), v\rangle
\]
for all $f \in \Ind_K^GV$,  $g \in G$ ($v$ is the same as in the definition of $\psi$).
 The following projection formula will be a very  useful tool in many occasions.
\begin{lemma}\label{lemmaD5}
If $v\in V$ has norm 1, then  we have 
\begin{equation}\label{eD5}
\sum_{k \in K}\langle \theta(k) u, v\rangle \theta(k^{-1})v = \frac{|K|}{d_\theta}u,
\end{equation}
for all $u \in V$.
\end{lemma}
\begin{proof}
Let $\{v_1, v_2, \ldots, v_{d_\theta}\}$ be an orthonormal basis of $V$ with $v = v_1$. Then, for $i,j = 1,2, \ldots, d_\theta$,
we have 
\[
\begin{split}
\left\langle \sum_{k \in K}\langle \theta (k)v_i, v\rangle\theta(k^{-1})v, v_j\right\rangle & = 
\sum_{k \in K}\langle \theta(k) v_i, v_1\rangle \overline{\langle \theta(k) v_j, v_1\rangle}\\
\mbox{(by \eqref{ORT})} \ \ \ \ \ \ \  & =\frac{|K|}{d_\theta} \delta_{i,j}.
\end{split}
\]
Therefore we have proved  \eqref{eD5} when $u = v_i$ and the general case follows  by linearity.
\end{proof}

\begin{proposition}\label{pD2}
\begin{enumerate}
\item{The operator $T_v$ belongs to $\Hom_G(\Ind_K^GV, L(G))$ and it is an isometry.}
\item{The operator $P: L(G) \longrightarrow L(G)$, defined by setting
\[
Pf = f*\psi
\]
for all $f \in L(G)$,  is  the orthogonal projection of $L(G)$ onto $T_v[\Ind_K^GV]$.}
\end{enumerate}
\end{proposition}
\begin{proof}
(1) The first part is obvious: it is immediate to check that 
\[
T_v\lambda(g) f = \lambda_G(g)T_v f.
\]
We now show that  $T_v$ is an isometry:  using the basis in \eqref{orthbasisind}  and assuming that $v = v_1$  we have, for $t_1,t_2 \in \mathcal{T}$, $i,j = 1, 2, \ldots, d_\theta$,
\[
\begin{split}
 \langle T_v\lambda(t_1)f_{v_i}, T_v\lambda(t_2)f_{v_j}\rangle _{L(G)} & = 
\frac{d_\theta}{|K|}\sum_{g \in G}\langle f_{v_i}(t_1^{-1}g), v\rangle\overline{\langle f_{v_j}(t_2^{-1}g), v\rangle}\\
\mbox{(by \eqref{definizione stellata1})} \ \ \ \  & = \frac{d_\theta}{|K|}\delta_{t_1, t_2}\sum_{k  \in K}\langle \theta(k^{-1}) v_i, v_1\rangle \overline{\langle \theta(k^{-1})v_j, v_1\rangle}\\
\mbox{(by \eqref{ORT})} \ \ \ \  & = \delta_{t_1, t_2}\delta_{i,j}.
\end{split}
\]

(2)  First of all,  note that 
\[
\psi * \psi = \psi \ \ \ \ \ \ \mbox{ and } \ \ \ \ \overline{\psi(g^{-1 })} = \psi(g).
\]

The first identity follows from \eqref{CON} applied to $\theta$ and ensures that $P$ is an idempotent; from the second identity we deduce that $P$ is selfadjoint, and therefore it is an orthogonal projection. Moreover, 
for all $f \in \Ind_K^GV,$ $g \in G$ we have 
\[
\begin{split}
[(T_vf)*\psi](g) & = \left(\frac{d_\theta}{|K|}\right)^{3/2}\sum_{k \in K} \langle f(gk^{-1}), v\rangle
\langle v, \theta(k) v\rangle\\
\mbox{(by \eqref{H31})} \ \ \ \ \ \ \ & = \left(\frac{d_\theta}{|K|}\right)^{3/2}\left\langle f(g), \sum_{k \in K} \langle  \theta(k)v,v \rangle \theta(k^{-1})v \right\rangle\\
\mbox{(by \eqref{eD5})} \ \ \ \ \ \ \  & = (T_vf)(g)
\end{split}
\]
that is,  $PT_v f = T_v f$ for all $f \in \Ind_K^GV$ (and in particular  $\ran P\supseteq T_v\Ind_K^GV$) . Finally, let us show that the range of $P$ is contained in (and therefore equal to) $T_v[\Ind_K^GV]$. Indeed, for all $\phi \in L(G)$, $g\in G$,
\[
\begin{split}
P\phi(g) & =  \sum_{k \in K}\phi(gk)\psi(k^{-1}) \\
& = \frac{d_\theta}{|K|}\left\langle\sum_{k\in K}\phi(gk)\theta(k)v, v\right\rangle\\
& = T_vf(g)
\end{split}
\]
if  $f(g) = \sqrt{\frac{d_\theta}{|K|}} \sum_{k \in K}\phi(gk)\theta(k)v$. Since it is immediate to check that $f$ belongs to $\Ind_K^G V$, we conclude that $\ran P = T_v\left[\Ind_K^GV\right]$.

\end{proof}

\begin{remark}{\rm
We   want  to relate the operator $T_v$ in the context of the results in Section 3. First note that the choice of $v$
is equivalent to the choice of an isometry $L \in \Hom(\CC,V)$, namely  
$L(\alpha) = \alpha v$ for $\alpha \in \CC$. Then, with $K,V, G, W$ replaced by ${1_K}, \CC,  K, V$ we have:

\[
[(L^*)^ \vee u](k) = \frac{1}{\sqrt{|K|}}\langle u, \theta(k)v\rangle
\]
for all $u \in V$ and $k \in K$ (clearly, $L^*u = \langle u,v\rangle$ for all $u \in V$). In particular, the map $S_v =  \sqrt{d_\theta}(L^*)^\vee$ is an isometric immersion of $V$ into $L(K)$ (this fact is also an easy consequence of the  orthogonality  relation for matrix coefficients). Considering  $S_v$  also as a map  in $\Hom_K(V, \Res_K^G L(G))$
(because $L(K) \subseteq \Res_K^GL(G)$ in the natural way), it is easy to prove that  $T_v = \sqrt{\frac{|G|}{|K|}}\overset{\diamond}{S_v}$, where  in this case we apply the machinery in Section 3 with $K, V$ and $G$ as in that section, but  with $W$ replaced by $L(G)$. Indeed, 
for $f \in \Ind_K^G V$ we have:
\[
\begin{split}
\sqrt{\frac{|G|}{|K|}}[\overset{\diamond}{S_v}f ](g)& =\left[\sum_{t\in \mathcal{T}}\lambda_G(t)S_v f(t)\right](g)\\
& = \sum_{t\in \mathcal{T}}[S_v f(t)](t^{-1}g)\\
& = \frac{{\sqrt{|d_\theta|}}} {\sqrt{|K|}}\sum_{\substack{t\in \mathcal{T}:\\ t^{-1}g \in K}}\langle f(t), \theta (t^{-1}g)v\rangle\\
\mbox{(setting $k_g = t^{-1}g$)} \ \ \ \ \ & = \frac{{\sqrt{|d_\theta|}}} {\sqrt{|K|}}\langle f(gk_g^{-1}), \theta (k_g)v\rangle\\
& = \frac{{\sqrt{|d_\theta|}}} {\sqrt{|K|}}\langle  f(g), v\rangle\\
& = T_vf(g).
\end{split}
\]}
\end{remark}

In the terminology of \cite{CURFOS, CR2, st4}, $\psi$ is a primitive idempotent,  $S_v(V) = \{f \in L(K): f *\psi = f\}$ is the minimal left ideal in $L(K)$ generated by $\psi$  and $T_v[\Ind_K^GV]$ is generated by $\psi$ as a left ideal in $L(G)$.
Then, the {\it Hecke algebra} $\mathcal{H}(G,K, \psi)$ is by definition
\[
\mathcal{H}(G,K, \psi) = \{\psi*f*\psi: f \in L(G)\} \equiv\{f \in L(G): f = \psi*f*\psi\}.
\]
It is well known that $\mathcal{H}(G,K,\psi)$ is antiisomorphic to $\Hom_G(\Ind_K^GV, \Ind_K^GV)$: we now want to  go further  and develop a suitable harmonic analysis in $\mathcal{H}(G,K,\psi)$, by translating the results of Section 4.\\

Now we introduce a suitable orthonormal basis in each $G$-irreducible representation. We divide the description of these bases in various cases.

Suppose that $\sigma\in J$ and that $\text{Res}_K^GW_\sigma=m_\sigma V\oplus\left(\oplus_{\rho\in R}m_\rho U_\rho\right)$ is the decomposition of $\text{Res}_K^GW_\sigma$ into irreducible $K$-representations, where $R$ contains the representations different from $\sigma$. Let $\{L_{\sigma,1},L_{\sigma,2},\dotsc,L_{\sigma,m_\sigma}\}$ be as in \eqref{choicebasis} and $v$ as above. We begin by introducing an orthonormal basis in the $V$-{\em isotypic component}. The first step consists in setting

\[
w_i^\sigma=L_{\sigma,i}v,\qquad i=1,2,\dotsc,m_\sigma.
\]
In the second and last step we introduce (see also Lemma  \ref{lemmaD5}) an orthonormal basis $v_1,v_2,\dotsc,v_{d_\theta}$ of $V$ with $v=v_1$ and we suppose that $\{w_h^\sigma:m_\sigma+1\leq h\leq m_\sigma d_\theta\}$ is an {\em arbitrary arrangement} of the vectors $\{L_{\sigma,i}v_j:1\leq i\leq m_\sigma,2\leq j\leq d_\theta\}$. The final result 

\[
\{w_h^\sigma:1\leq h\leq m_\sigma d_\theta\}
\] 
is the desired orthonormal basis in $m_\sigma V$.\\

Then we repeat the construction for each $U_\rho,\rho\in R$, without an initial choice of a vector in $U_\rho$ (we avoid the first step): we select an orthonormal basis $\{M_{\rho,1},M_{\rho,2},\dotsc,M_{\rho,m_\rho}\}$ for $\text{Hom}_K(U_\rho,\text{Res}_K^GW_\sigma)$, an orthonormal basis $\{u^\rho_1,u^\rho_2,\dotsc,u_{d_\rho}^\rho\}$ in $U_\rho$ and we suppose that

\[
\{w_h^\sigma: m_\sigma d_\theta+1\leq h\leq d_\sigma\}
\]
is an {\em arbitrary arrangement} of the vectors $\{M_{\rho,i}u_j^\rho:\rho\in R, 1\leq i\leq m_\rho,1\leq j\leq d_\rho \}$. The final result is an orthonormal basis $\{w_h^\sigma:1\leq h\leq d_\sigma\}$ for $W_\sigma$: we say that it is {\em adapted to the choice of} $v$ {\em and of} $\{L_{\sigma,1},L_{\sigma,2},\dotsc,L_{\sigma,m_\sigma}\}$.\\

If $\sigma\notin J$ then $\{w_h^\sigma:1\leq h\leq d_\sigma\}$ is an {\em arbitrary} orthonormal basis of $W_\sigma$. The importance of such bases is in the following property.\\

\begin{lemma}\label{lD10}
\begin{enumerate}
\item{If $\sigma \in J$, $1\leq j \leq m_\sigma$ and  $1\leq h\leq d_\sigma$ then
\begin{equation}\label{stellap16}
L^*_{\sigma,j}w^\sigma_h = \left\{\begin{array}{ll}
v_\ell  & \mbox{if \  $w_h^\sigma = L_{\sigma,j}v_\ell$ \  for some \  $\ell \in \{1,2, \ldots, d_\theta\}$}\\
0 & \mbox{otherwise.}
\end{array}
\right.
\end{equation}}
\item{
\begin{equation}\label{stellap162}\sum_{k \in K}\psi(k)\sigma(k)w_i^\sigma = 
\left\{\begin{array}{ll}
\frac{|K|}{d_\theta}w_i^\sigma &  \mbox{ if } \sigma \in J, \  1 \leq i \leq m_\sigma\\
0 & \mbox{ otherwise.}
\end{array}
\right.
\end{equation}}
\end{enumerate}
\end{lemma}
\begin{proof}
(1) This is a consequence of  \eqref{orthdeltaij}  and the definition of the vectors $w_h^\sigma$.
(2) First of all, note that 
\[
\sum_{k \in K}\psi(k)\sigma(k)w_i^\sigma = \frac{d_\theta}{|K|}\sum_{k \in K}\langle \theta(k) v, v\rangle\sigma(k^{-1})w_i^\sigma.
\]
 If $\sigma\in J$ and $1 \leq i \leq m_\sigma$, then we may apply \eqref{eD5} since 
$\sigma(k^{-1 })w_i^\sigma = \sigma(k^{-1})L_{\sigma,i}v = L_{\sigma,i}\theta(k^{-1})v$.
Otherwise, we can argue as in the proof of \eqref{eD5}: the bases are adapted to the choice of $v$ and to the decomposition of $\Res_K^G W_\sigma$ and therefore  we may use the orthogonality relations for the matrix coefficients in $L(K)$. For instance, if  $\sigma \in J$ and $m_\sigma< i\leq m_\sigma d_\theta$ then $w_i^\sigma = L_{\sigma,h}v_j$ for some $1\leq h \leq m_\sigma$, $2\leq j \leq d_\theta$, and therefore, for all $\ell = 1,2, \ldots, d_\sigma$
\[
\begin{split}
\left\langle \sum_{k \in K}\langle \theta(k)v,v\rangle\sigma(k^{-1})L_{\sigma,h}v_j,w_{\ell}^\sigma\right\rangle & =  \sum_{k \in K}\langle \theta(k)v,v\rangle\langle \theta(k^{-1})v_j, L^*_{\sigma,h}w_\ell^\sigma\rangle\\
&= 0
\end{split}
\]
because if  $ L^*_{\sigma,h}w_\ell^\sigma\neq 0$  by \eqref{stellap16} it is equal   to one of the $v_1, v_2, \ldots, v_{d_\theta}$ and $j \geq 2$. 
In the other cases we are dealing with  matrix coefficients corresponding to inequivalent  $K$-representations.

\end{proof}

In what follows, we denote by $u_{i,j}^\sigma$ the matrix coefficients of $\sigma$  corresponding to the bases chosen above,
that is,
\[
u_{i,j}^\sigma(g) = \langle \sigma(g)w_j^\sigma, w_i^\sigma\rangle
\]
for $\sigma\in  \widehat{G}$, $ i,j = 1,2, \ldots, d_\sigma$, $g \in G$.  Then we define the convolution operators
($\sigma \in J$, $i,j = 1, 2, \ldots, m_\sigma$):
\[
\widetilde{U}^\sigma_{i,j} f = \frac{d_\sigma}{|G|}f*\overline{u^\sigma_{j,i}}
\]
for all $f \in L(G)$. We want to show that $\widetilde{U}^\sigma_{i,j}$ corresponds to $U_{i,j}^\sigma$ in \eqref{quad11} under the isometry $T_v$.

\begin{theorem}\label{tD12}
For all $\sigma \in J$,  $i,j = 1, 2, \ldots, m_\sigma$ and $f \in \Ind_K^GV$ we have:
\[
T_vU^\sigma_{i,j} f = \widetilde{U}^\sigma_{i,j}T_vf,
\]
that is, the following diagram is commutative:
 \[
\begin{array}{ccc}
\Ind_K^G V&\stackrel{T_v}{\longrightarrow}
&T_v[\Ind_K^G V]\\
U_{i,j}^\sigma\downarrow &   &  \downarrow \widetilde{U}^\sigma_{i,j}  \\
\Ind_K^G V&\stackrel{T_v}{\longrightarrow}
&T_v[\Ind_K^G V].\\
\end{array}
\]
\end{theorem}
\begin{proof}
For all $g \in G$ we have:
\[
\begin{split}
[T_vU^\sigma_{i,j} f](g) & = \frac{d_\sigma}{\sqrt{|K|d_\theta}}\langle[(L^*_{\sigma,i})^\vee\overset{\diamond}{L}_{\sigma,j}f](g), v\rangle\\
%& = \frac{d_\sigma}{\sqrt{|G|d_\theta}}\langle L^*_{\sigma,i}\sigma(g^{-1})\overset{\diamond}{L}_{\sigma,j}f, v\rangle\\
& =  \frac{d_\sigma}{|G|\sqrt{|K|d_\theta}}\sum_{g_1 \in G}
\langle L^*_{\sigma,i}\sigma(g^{-1}g_1)L_{\sigma,j}f(g_1),  v\rangle\\
& =  \frac{d_\sigma}{|G|\sqrt{|K|d_\theta}}\sum_{g_1 \in G}\langle L_{\sigma,j}f(g_1), \sigma(g_1^{-1}g) w_i^\sigma\rangle\\
& =   \frac{d_\sigma}{|G|\sqrt{|K|d_\theta}}\sum_{g_1 \in G}\sum_{h = 1}^{d_\sigma}\langle f(g_1), L^*_{\sigma,j}w_h^\sigma\rangle \overline{\langle\sigma(g_1^{-1}g)w_i^\sigma, w_h^\sigma\rangle}\\
\mbox{(by  \eqref{stellap16})} \ \ \ \  & = \frac{d_\sigma}{|G|\sqrt{|K|d_\theta}}\sum_{g_1 \in G}\sum_{\ell = 1}^{d_\theta}\langle f(g_1), v_\ell\rangle\overline{ \langle \sigma(g_1^{-1}g)w_i^\sigma, L_{\sigma, j}v_\ell\rangle}\\
\mbox{(by  \eqref{eD5} with $u = v_\ell$)} \ \ \ \  & = \frac{d_\sigma}{|G|}\frac{\sqrt{d_\theta}}{|K|^{3/2}}\sum_{\ell = 1}^{d_\theta}\sum_{g_1 \in G}\sum_{k \in K}\langle f(g_1), \theta(k^{-1}) v\rangle \cdot \\
& \ \ \ \ \ \cdot  \langle v, \theta(k)v_\ell\rangle
\overline{\langle\sigma(g_1^{-1}g)w_i^\sigma, L_{\sigma, j}v_\ell\rangle}\\
\mbox{($g_1 = g_2k$)}  \ \ \ \  & = \frac{d_\sigma}{|G|}\frac{\sqrt{d_\theta}}{|K|^{3/2}}\sum_{\ell = 1}^{d_\theta}\sum_{g_2 \in G}\langle f(g_2), v\rangle\cdot \\
& \ \ \ \ \ \ \ \cdot  \overline{\left\langle \sigma(g_2^{-1}g)w_i^\sigma, L_{\sigma,j}\sum_{k \in K} \langle \theta(k^{-1}) v, v_\ell\rangle \theta(k)v_\ell\right\rangle}   \\
   & =   \frac{d_\sigma\sqrt{d_\theta}}{|G|\sqrt{|K|}}\sum_{g_2 \in G}\langle f(g_2), v\rangle \overline{\langle \sigma(g_2^{-1}g)w_i^\sigma, w_j^\sigma\rangle}\\
   & \mbox{(by  \eqref{eD5} with $u , v$ replaced by $v, v_\ell$)}\\
  & =   \frac{d_\sigma \sqrt{d_\theta}}{|G|\sqrt{|K|}}\sum_{g_2 \in G}\langle f(g_2), v\rangle \overline{u^\sigma_{j,i}(g_2^{-1}g)}\\
& = [\widetilde{U}_{i,j}T_vf](g).
\end{split}
\]
\end{proof}
The following lemma shows that the only matrix coefficients $u_{i,j}^\sigma$ in  $\mathcal{H}(G,K,\psi)$ are the conjugate of those that come from the $V$-isotypic component of $W_\sigma$.
\begin{lemma}\label{lD15}
We have 
\[
\psi*\overline{u_{i,j}^\sigma}* \psi = \left\{
\begin{array}{ll}
\overline{u^\sigma_{i,j}} & \mbox{ if $\sigma \in J$ and $1 \leq i,j \leq m_\sigma$}\\
0 & \mbox{otherwise.}
\end{array}
\right.
\]
\end{lemma}
\begin{proof}

We have, for all $g \in G$, 
\begin{equation}\label{stellap18}
\begin{split}
\overline{u^\sigma_{i, j}}* \psi(g) &  = \sum_{k \in K}u_{j,i}^\sigma(kg^{-1})\psi(k)\\
\mbox{(by definition of  $u_{i,j}^\sigma$)} \ \ \ \ \  &  =\left \langle \sigma(g^{-1})w_i^\sigma, \sum_{k \in K}\psi(k) \sigma(k) w_j^\sigma\right\rangle\\
\mbox{(by \eqref{stellap162})} \ \ \ \ \  &  =
 \left\{
\begin{array}{ll}
\overline{u^\sigma_{i,j}}(g) & \mbox{ if $\sigma \in J$ and $1 \leq j \leq m_\sigma$}\\
0 & \mbox{otherwise.}
\end{array}
\right.
\end{split}
\end{equation}
The reader can complete the proof by computing in a similar way $\psi* \overline{u^\sigma_{i,j}}$.

\end{proof}

By virtue of Theorem  \ref{tD12} and Lemma \ref{lD15} it is convenient to set:
\[
\phi^\sigma_{i,j} = \overline{u^\sigma_{i,j}}
\]
for $\sigma \in J$, $i,j = 1,2, \ldots, m_\sigma$. In other words,
$$\phi^\sigma_{i,j}(g) = \langle w_i^\sigma, \sigma(g)w_j^\sigma\rangle$$ for all  $g \in G$.
Compare with Definition 9.4.5. of \cite{book}  and Definition 2.10 of \cite{st1} (see also \cite{ Macdonald, Terras}).

If $f \in \mathcal{H}(G,K,\psi)$, its  {\it Fourier transform} at $\sigma \in J$ is the $m_\sigma \times m_\sigma$ complex matrix whose
$i,j$-entry is $$[\mathcal{F}(\sigma)]_{i,j} = \langle f, \phi_{i,j}^\sigma \rangle_{L(G)}$$ for $\sigma \in J$ and $i,j = 1,2, \ldots, m_\sigma$. As in \eqref{2quad11}, we set $\mathcal{F} f= \bigoplus_{\sigma \in J}\mathcal{F}f(\sigma)$.
We denote by $\chi^\sigma$ the character of the $G$-irreducible  representation  $(\sigma, W_\sigma)$. Moreover, 
$\chi^\sigma(f) = \sum_{g \in G}\chi^\sigma(g)f(g)$  for all $f \in L(G)$.  

\begin{theorem}\label{tD17}
\begin{enumerate}
\item{The set $\{\phi^\sigma_{i,j}: \sigma \in J, i,j = 1,2, \ldots, m_\sigma\}$ is an orthogonal basis for 
$\mathcal{H}(G,K,\psi)$  and $\|\phi^\sigma_{i,j}\|^2_{L(G)} = \frac{|G|}{d_\sigma}$.  In particular, the Fourier inversion formula is
\[
f = \frac{1}{|G|}\sum_{\sigma\in J}d_\sigma\sum_ {i,j = 1}^{m_\sigma}\left[\mathcal{F}(\sigma)\right]_{i,j}\phi^\sigma_{i,j}.
\]}
\item{The map
\[
\begin{array}{ccc}
 \mathcal{H}(G,K,\psi) & \longrightarrow & \bigoplus_{\sigma \in J}M_{m_\sigma, m_\sigma}(\CC)\\
f& \longmapsto & \mathcal{F}f
\end{array}
\]
is an isomorphism of algebras.}
\item{Set $$\phi^\sigma = \sum_{i= 1}^{m_\sigma}\phi_{i,i}^\sigma$$  and suppose  that 
$\varphi^\sigma$ is the irreducible  character of $ \mathcal{H}(G,K,\psi)$ corresponding to $M_{m_\sigma, m_\sigma}(\CC)$. Then 
\begin{equation}\label{stellap19}
\varphi^\sigma(f) = \chi^\sigma(f) = \sum_{g \in G}f(g) \overline{\phi^\sigma(g)},
\end{equation}
for all $f \in  \mathcal{H}(G,K,\psi)$.
Moreover, the  $\phi^\sigma$'s satisfy the following orthogonality relations:
\begin{equation}\label{ORT2}
\langle \phi^\sigma, \phi^\rho\rangle = \delta_{\sigma,\rho}\frac{|G|m_\sigma}{d_\sigma}.
\end{equation}}
\end{enumerate}
\end{theorem}
\begin{proof}
(1) It follows  from Lemma \ref{lD15} and the usual orthogonality relations for matrix coefficients in \eqref{ORT}.

(2) The usual convolution properties of the matrix coefficients \eqref{CON}  yields
\[
\phi^\sigma_{i,j}*\phi^\rho_{h,\ell} = \frac{|G|}{d_\sigma} \delta_{\sigma, \rho}\delta_{j,h}\phi^\sigma_{i,\ell}
\]
and this, combined with the Fourier inversion formula, implies  that the Fourier transform is an isomorphism.

(3) For  all $f  \in  \mathcal{H}(G,K,\psi)$ and $\sigma \in J$, we have 
\[
\begin{split}
\varphi^\sigma(f)  & = \sum_{i= 1}^{m_\sigma}\left[\mathcal{F} f(\sigma)\right]_{i,i}
 =  \sum_{i= 1}^{m_\sigma}\langle f, \phi^\sigma_{i,i}\rangle
  = \sum_{g \in G}f(g)\overline{\phi^\sigma(g)}\\
& =  \sum_{i= 1}^{m_\sigma}\sum_{g \in G}f(g) u_{i,i}^\sigma(g) = \chi^\sigma(f),
\end{split}
\]
where the last equality follows from the fact that $\langle f,  u^\sigma_{i,i}\rangle = 0$ if $i>m_\sigma$.
The proof of \eqref{ORT2} is obvious.
\end{proof}
\begin{remark}{\rm If $\phi \in L(G)$ the convolution operator $T_\phi:L(G) \to L(G)$ associated with  $\phi$ is defined by setting $T_\phi f = f*\psi$  for all $f \in  L(G)$.  Since $T_{\phi_1 * \phi_2} = T_{\phi_2}T_{\phi_1}$, 
the map $\phi \mapsto T_\phi$ is an {\it antiisomorphism}  between $L(G)$ and $\Hom_G(L(G),L(G))$,  see 
Exercise 4.2.2.  in \cite{book}. It folllows that the map $U_{i,j}^\sigma \mapsto \widetilde{U}_{i,j}^\sigma$ in Theorem \ref{tD12} yields the {\it antiisomorphism}   $U^\sigma_{i,j} \mapsto \phi_{j,i}^\sigma$  between $\Hom_G(\Ind_K^G V, \Ind_K^G V)$ and $\mathcal{H}(G,K,\psi)$ (see also \cite{CURFOS,CR2, st4}). Actually, these algebras are isomorphic, because they are both isomorphic to 
$\bigoplus_{\sigma \in J}M_{m_\sigma,m_\sigma}(\CC)$.
Moreover, all other results in Theorem \ref{teorema5} may be translated in the present setting.  For instance, if we define 
$\widetilde{T}_{\sigma,i}:W_\sigma \to L(G)$ by setting
\[
(\widetilde{T}_{\sigma,i} w)(g) = \sqrt{\frac{d_\sigma}{|G|}}\langle w, \sigma(g)w^\sigma_i\rangle
\]
for all $w \in W$, $g \in G$, then it is easy to check that $T_vT_{\sigma,i} = \widetilde{T}_{\sigma,i}$ and  
$T_v[\Ind_K^G V] = \bigoplus_{\sigma \in J}\bigoplus_{i = 1}^{m_\sigma} \widetilde{T }_{\sigma,i} W_\sigma$  is an isometric orthogonal decomposition. Moreover, $\widetilde{U}_{i,j}^\sigma$ intertwines $\widetilde{T}_{\sigma,j}W_\sigma$ with $\widetilde{T}_{\sigma,i}W_\sigma$ and $\widetilde{U}_{i,i}^\sigma$ is the orthogonal projection onto $\widetilde{T}_{\sigma,i}W_\sigma$}.
\end{remark}

We now prove some formulas that relate  $\chi^\sigma, \varphi^\sigma$ and $\phi^\sigma$ (see also \eqref{stellap19}).  We recall that $\delta_g$ is the  
Dirac function centered at $g$, that is  $\delta_g(g_0) = \left\{\begin{array}{ll}1 & \mbox{ if  $g = g_0$}\\ 0 & \mbox{otherwise.}
\end{array}\right.$

\begin{theorem}\label{tD23} We have:
\begin{equation}\label{e1D23}
\chi^\sigma(g) = \frac{d_\sigma}{|G|m_\sigma}\sum_{h \in G}\overline{\phi^\sigma(h^{-1}gh)}
\end{equation}
and 
\begin{equation}\label{e2D23}
\phi^\sigma(g) =\overline{\varphi^\sigma(\psi*\delta_g*\psi)}
\end{equation}
for all $\sigma\in J$ and $g \in G$.
\end{theorem}
\begin{proof}
The proof of \eqref{e1D23} is easy: it follows from \eqref{eD21}, taking into account that $\phi^\sigma$ is the sum of the conjugate of $m_\sigma$ diagonal matrix coefficients. 

We now turn to the proof of  \eqref{e2D23}.  Starting from \eqref{stellap19}  we  get:
\[
\begin{split}
\varphi^\sigma(\psi*\delta_g*\psi) & = 
\sum_{g_1 \in G}(\psi * \delta_g*\psi)(g_1)\overline{\phi^\sigma(g_1)}\\
& = \sum_{\substack{g_1 \in G: \\g_1k_2^{-1}g^{-1} \in K}}\sum_{k_2 \in K}\left[
\psi(g_1k_2^{-1}g^{-1})\psi(k_2)\right]\overline{\phi^\sigma(g_1)}\\
\mbox{( setting  $k_1 = g_1k_2^{-1}g^{-1})$}\ \ \ \  & = \sum_{k_1,k_2\in K}\psi(k_1)\psi(k_2)\overline{\phi^\sigma(k_1gk_2)}\\
\mbox{(by Lemma \ref{lD15})} \ \ \ \ \  & = \overline{\phi^\sigma(g)}.
\end{split}
\]
\end{proof}

\begin{remark}
{\rm In \cite{CURFOS} (see also \cite{CR2}) it is proved a  formula  that expresses $\chi^\sigma$ in terms of $\varphi^\sigma$. In our notation it reads:
\begin{equation}\label{eD26}
\chi^\sigma(g)  = \frac{|G|}{|\mathcal{C}(g)|}\varphi^\sigma(\psi* {\bf 1}_{\mathcal{C}(g)}* \psi)
\left[\sum_{h \in G}\varphi^\sigma(\psi*\delta_{h^{-1}}*\psi)\cdot\varphi^\sigma(\psi*\delta_h*\psi)\right]^{-1}
\end{equation}
where $\mathcal{C}(g)$ denotes the conjugacy class of $g \in G$  and $ {\bf 1}_{\mathcal{C}(g)}$ its  characteristic function.
In the research-expository paper \cite{st4} it is showed, among other things,  how to obtain these results  using the techniques  in \cite{CURFOS} and \cite{CR2}.
%there is a complete exposition on these results using the methods in \cite{CURFOS} and \cite{CR2}. 
We now want to deduce  \eqref{eD26} from the results proved in the present paper. 
First note that by \eqref{e2D23}
\begin{equation}\label{e2D26}
\begin{split}
\sum_{h \in G}\varphi^\sigma(\psi*\delta_h*\psi)\varphi^\sigma(\psi*\delta_{h^{-1}}*\psi) & = 
\sum_{h \in G}\overline{\phi^\sigma(h)}\overline{\phi^\sigma(h^{-1})}\\
& = \sum_{h \in G}|\phi^\sigma(h)|^2\\
\mbox{ (by \eqref{ORT})}\ \ \ \ \ & = \frac{m_\sigma |G|}{d_\sigma}.
\end{split}
\end{equation}

Moreover, starting from the equality $ {\bf 1}_{\mathcal{C}(g)} = \frac{|\mathcal{C}(g)|}{|G|} \sum_{h \in G}\delta_{h^{-1}gh}$
we get
\begin{equation}\label{eD27}
\begin{split}
\varphi^\sigma(\psi*{\bf 1}_{\mathcal{C}(g)}*\psi) & =  \frac{|\mathcal{C}(g)|}{|G|}\sum_{h \in G}
\varphi^\sigma(\psi* \delta_{h^{-1}gh}*\psi)\\
\mbox{(by \eqref{e2D23})} \ \ \ \ \ & =   \frac{|\mathcal{C}(g)|}{|G|}\sum_{h \in G}\overline{\phi^\sigma(h^{-1}gh)}\\
\mbox{(by \eqref{e1D23})} \ \ \ \ \ & =    \frac{|\mathcal{C}(g)|m_\sigma}{d_\sigma}\chi^\sigma(g).
\end{split}
\end{equation}
Then \eqref{eD26} follows from \eqref{e2D26} and \eqref{eD27}.}
\end{remark}
The spherical functions of a finite Gelfand pair satisfy  the following  functional identity 
\[
\frac{1}{|K|}\sum_{k\in K}\phi(gkh) = \phi(g)\phi(h)
\]
for all $g,h \in G$ (see Theorem 4.5.3 in \cite{book} and \cite{ Faraut, FN, Macdonald, Terras}).  We give  an analogous identity for the  matrix coefficients $\phi_{i,j}^\sigma$.
\begin{proposition}
\label{pennap1}
For $\sigma \in J$,  $i,j = 1,2, \ldots,m_\sigma$ and $g,h \in G$, we have 
\[
\sum_{k \in K}\phi_{i,j}^\sigma (gkh)\overline{\psi(k)} = \sum_{\ell= 1}^{m_\sigma}\phi_{i,\ell}^\sigma(g)\phi_{\ell,j}^\sigma(h).
\]
\end{proposition}
\begin{proof}
\[
\begin{split}
\sum_{k \in K}\phi_{i,j}^\sigma(gkh)\overline{\psi(k)}   & = \sum_{k \in K}\langle w_i^\sigma, \sigma(gkh)w_j^\sigma\rangle\overline{\psi(k)} \\
& = \sum_{\ell = 1}^{d_\sigma}\langle \sigma(g^{-1  })w_i^\sigma, w_\ell^\sigma\rangle\sum_{k\in K}\overline{u_{\ell,j}^\sigma(kh)}\overline{\psi(k)} \\
\mbox{(by \eqref{stellap18})}\ \ \ \ \  &= \sum_{\ell= 1}^{m_\sigma}\phi_{i,\ell}^\sigma(g)\phi_{\ell,j}^\sigma(h).
\end{split}
\]
\end{proof}
\begin{remark}{\rm If the $K$-representation $(\theta,V)$ is one dimensional, it can be identified with its character 
$\chi: K \to \{z \in \CC: |z| = 1\}$ which  satisfies the identity
$\chi(k_1k_2) = \chi(k_1)\chi(k_2)$, for all $k_1,k_2 \in K$.  
It follows that 
\[
\mathcal{H}(G,K,\psi) = \{f \in L(G): f(k_1gk_2) = \overline{\chi(k_1)}\overline{\chi(k_2) }f(g), \ \forall k_1,k_2 \in K , \ g \in G\}.
\]
See \cite{st4} for the easy details. If $G = \coprod_{s \in \mathcal{S}}KsK$ is the decomposition of $G$ into double $K$-cosets, then
a function $f \in \mathcal{H}(G,K,\psi)$ is determined by its values on $\mathcal{S}$. In particular, the  orthogonality relations for $\phi_{i,j}^\sigma$ and $\phi^\sigma$ take the form:
\[
\sum_{s \in \mathcal{S}}|KsK|\phi_{i,j}^\sigma(s)\overline{\phi^\rho_{\ell,r}(s)}= \frac{|G|}{d_\sigma}\delta_{\sigma,\rho}\delta_{i, \ell}\delta_{j,r}
\]
and
\[
\sum_{s \in \mathcal{S}}|KsK|\phi^\sigma(s)\overline{\phi^\rho(s)} = \frac{m_\sigma|G|}{d_\sigma}\delta_{\sigma,\rho}.
\]
>From the last formula, it is just an easy exercise to deduce the orthogonality relations in 2, Theorem 4.24 in \cite{st4}, originally 
proved by Curtis and Fossum (see Theorem 2.4  in \cite{CURFOS} or (ii), Theorem 11.32 in \cite{CR2}).
 }
\end{remark}

\section{Gelfand-Tsetlin bases}\label{SecGelTse}
We now extend to our setting the classical theory of Gelfand-Tsetlin bases (cf.  \cite{book2, OV,st2}), that yields a natural choice for the orthonormal basis in Corollary \ref{corollario1}. 
We continue to use the notation of the previous sections (in particular Sections \ref{SecorthFrob} and \ref{Seccommutant}). 
First we prove a preliminary result that examines  the correspondence $L \mapsto (L^*)^\vee$ in relation to the induction in stages.  Let $H$ be a subgroup of $G$ containing $K$ (i.e. $K \leq H\leq G$) and denote by $(\rho,U)$ an irreducible  $H$-representation. 
If $L_1 \in \Hom_K(V, \Res^H_K U)$ and $L_2\in \Hom_H(U,\Res_H^GW)$ then  $L_2L_1 \in \Hom_K(V, \Res_K^GW)$. 
Since $(L_1^*)^\vee \in \Hom_H(U, \Ind_K^HV)$, we can consider $(L_1^*)^\vee U$ as a subspace of $\Ind_K^HV$. Therefore, $\Ind_H^G[(L_1^*)^\vee U]$ is a subspace of 
\begin{equation}\label{indstage}
\Ind_H^G[\Ind _K^HV]\cong  \Ind_K^GV.
\end{equation}
We recall the construction of this isomorphism, which is the transitive property of the induction  (cf. \cite{ Mackey,book2}).
The Left Hand Side
% $\Ind_H^G[\Ind _K^HV]$
 may be seen as the set of all $F: G\times H \to V$ such that 
$F(gh, h_0k) =  \theta(k^{-1})F(g,hh_0)$, for all $g \in G$,  $h,h_0\in H$ and $k \in K$.  Being the Right Hand Side %$\Ind_K^G V$ 
as in \eqref{H31}, we have that the isomorphism in \eqref{indstage} is given by the map
\begin{equation}\label{indstage2}
F\mapsto f
\end{equation}
where $f(g) = F(g,1_G)$ for all $g \in G$ (note that $F$ is uniquely determined  by $f$, because $F(g,h) = f(gh)$ for all $g\in G$ and $h \in H$).

\begin{theorem}\label{GZt3}
Under the isomorphism \eqref{indstage2}, we have 
\[
\left[(L_2L_1)^*\right]^\vee W\leq \Ind_H^G\left[(L_1^*)^\vee U\right].
\]
\end{theorem}
\begin{proof}
The space $\Ind_H^G[(L_1^*)^\vee U]$ is made up of all functions $F\in \Ind_H^G[\Ind_K^HV]$ such that, for every fixed $g\in G$, the function $h\mapsto F(g,h)$ belongs to $(L_1^*)^\vee U$, i.e.  there exists $u_g \in U$ such that 
\begin{equation}\label{indstage3}
F(g,h) = \left[(L_1^*)^\vee u_g\right](h).
\end{equation}
For $w \in W$, $g\in G$ we have:
\[
\left\{\left[(L_2L_1)^*\right]^\vee w\right\}(g) = \frac{1}{\sqrt{|G/K|}}L_1^*L_2^*\sigma(g^{-1})w
\]
and therefore
\[
\begin{split}
\left\{\left[(L_2L_1)^*\right]^\vee w\right\}(gh) = & \frac{1}{\sqrt{|G/K|}}L_1^*L_2^*\sigma(h^{-1})\sigma(g^{-1}) w\\
\mbox{($L_2 \in \Hom_H(U,\Res_H^GW)$})\ \ \ \ \ \ \ & = \frac{1}{\sqrt{|G/K|}}L_1^*\rho(h^{-1})[L_2^*\sigma(g^{-1})w]\\
& = \frac{\sqrt{|H/K|}}{\sqrt{|G/K|}}\left\{(L_1^*)^\vee[L_2^*\sigma(g^{-1})w]\right\}(h)\\
& = \left\{(L_1^*)^\vee[(L_2^*)^\vee w](g)\right\}(h).
\end{split}
\]
This means that, with respect to \eqref{indstage2}, the function $f = {(L_2L_1)^*}^\vee w \in \Ind_K^GV$ corresponds to an $F(g,h)$ of the form \eqref{indstage3}, with $u_g =[ ({L_2^*)^\vee w}](g)$. Therefore,  $f\in \Ind_H^G[(L_1^*)^\vee U]$.
\end{proof}

Suppose now that there exists a chain of subgroups of $G$ of the form
\begin{equation}\label{diamond1}
K = H_1 \leq H_2\leq \cdots \leq H_{m-1} \leq H_m = G.
\end{equation}
Define recursevely $J_\ell\subseteq  \widehat{H_\ell}$, $1\leq \ell\leq m$, by setting $J_1 =\{\theta\}$ and 
$J_{\ell+1}$ equal to  the set of  all $\eta \in \widehat{H_{\ell+1}}$ such that $\eta$ is contained in $\Ind_{H_\ell}^{H_{\ell+1}}\rho$, for some $\rho\in J_\ell$, $\ell = 1,2,\ldots, m-1$. We say that the chain 
\eqref{diamond1} satisfies the {\it Gelfand-Tsetlin} condition if for all $1\leq \ell\leq m-1$, $\rho \in J_\ell$ and 
$\eta \in J_{\ell+1}$ the multiplicity of $\eta\in \Ind_{H_\ell}^{H_{\ell+1}}\rho$ (equivalently, the multiplicity of $\rho$ in $\Res_{H_\ell}^{H_{\ell+1}}\eta$) is at most 1;  we write $\eta\to \rho$ when the multiplicity is equal to 1.
If the Gelfand-Tsetlin condition is satisfied, the associated {\it Bratteli diagram} is the finite  oriented graph whose vertex set is $\coprod_{\ell =1}^mJ_\ell$ and the edge set is formed by the pairs $(\eta,\rho)$ such that $\eta \to \rho$.  A {\it path} in the Bratteli diagram is a sequence $C: \rho_m \to \rho_{m-1}\to \cdots\to \rho_2\to \rho_1$,
where $\rho_1 = \theta$ and $\rho_m\in J$. For every $\sigma \in J$, we denote by $\mathcal{P}(\sigma)$ the set of all paths $C: \rho_m \to \rho_{m-1}\to \cdots\to \rho_2\to \rho_1$ such that $\rho_m = \sigma$. Fix now $\sigma \in J$ and denote  by $W$ its representing space.  We define recursively a chain of subspaces
\[
W_m\geq W_{m-1}\geq \cdots\geq W_2\geq W_1
\]
as follows. We set $W_m = W$ and for $\ell = m-1, m-2, \ldots, 1$, we denote by $W_\ell$ the {\it unique} subspace of $\Res^{H_{\ell+1}}_{H_\ell}W_{\ell+1}$ isomorphic to the representation space of $\rho_\ell$. This way, $W_1 \sim V$ as a $K$-representation; we set $V_C = W_1$.  If $\widetilde{C}: \widetilde{\rho}_m\to \widetilde{\rho}_{m-1}\to \cdots\to \widetilde{\rho}_2\to \widetilde{\rho}_1$ is a different path in $\mathcal{P}(\sigma)$, then there exists  $2 \leq \ell\leq m$ such that $\rho_i \sim \widetilde{\rho}_i$ for $i = m, m-1, \ldots, \ell$  and 
$\rho_{\ell-1}\not\sim \widetilde{\rho}_{\ell-1}$. Therefore, if $\widetilde{W}_m\geq \widetilde{W}_{m-1}\geq \cdots\geq\widetilde{W}_2\geq\widetilde{W}_1$ is the chain of subspaces associated with $\widetilde{C}$ then  $W_i =\widetilde{W}_i$, $i = m, m-1, \ldots, \ell$,  but 
$W_{\ell-1}$ and $\widetilde{W}_{\ell-1}$ are orthogonal, because they afford inequivalent representations. This implies that also $V_C$ and $V_{\widetilde{C}}$ are orthogonal.  Finally, by induction on $m$, it is easy to prove that 
\begin{equation}\label{stellaX1}
\bigoplus_{C \in \mathcal{P}(\sigma)}V_C
\end{equation}
is an orthogonal decomposition of the $\theta$-isotypic component of $\Res^G_KW$. Let $L_{\sigma, C}\in \Hom_K(V,\Res^G_KW)$ be an isometry with $L_{\sigma,C}V = V_C$. The operator $L_{\sigma,C}: V \to W$ is defined up to a complex constant of modulus 1 (the {\it phase factor}) and, by Lemma \ref{lemma1} the set 
\begin{equation}\label{stellaX12}
\{L_{\sigma,C}: C \in \mathcal{P}(\sigma)\}
\end{equation}
is an orthonormal basis for $\Hom_K(V, \Res_K^G W)$.

Similarly, with each $C\in \mathcal{P}(\sigma)$, $C:\rho_m\to \rho_{m-1}\to \cdots, \rho_2 \to \rho_1$, we can 
associate  the following sequence of spaces: $Z_1 = V$, and recursively, $Z_{\ell+1}$ is the unique subspace of $\Ind_{H_\ell}^{H_{\ell+1}}Z_\ell$, that affords $\rho_{\ell+1}$; finally, we set $W_C= Z_m$. Clearly, $W_C$ is a subspace of $\Ind_K^G V$ and 
\begin{equation}\label{stellaX13}
\bigoplus_{C \in \mathcal{P}(\sigma)}W_C
\end{equation}
is an orthogonal decomposition of the $\sigma$-isotypic component of $\Ind_K^G V$. Indeed, we have 
\[
\Ind_K^GV = \Ind_{H_{m-1}}^{H_m}\Ind_{H_{m-2}}^{H_{m-1}}\cdots \Ind_{H_1}^{H_2} V
\]
and at each stage  the induction is {\it multiplicity free}.

We now show that the decomposition \eqref{stellaX1} and \eqref{stellaX13} are closely related as in Corollary \ref{corollario2}.
\begin{theorem}
The orthonormal basis
\[
\left\{\sqrt{\frac{d_\sigma}{d_\theta}}(L_{\sigma,C}^*)^\vee: C \in \mathcal{P}(\sigma)\right\}
\]
of $\Hom_G(W, \Ind_K^G V)$ gives rise precisely to the isometric orthogonal decomposition \eqref{stellaX13}, that is 
\[
W_C = \sqrt{\frac{d_\sigma}{d_\theta}}(L_{\sigma,C}^*)^\vee W
\]
for every $C \in \mathcal{P}(\sigma)$.
\end{theorem}
\begin{proof}
It follows from Corollary \ref{corollario2} and a repeated application of Theorem \ref{GZt3},  by induction on $m$.
\end{proof}

%%%%%%%%%%%%%%%%%%%%%%%%%%%%%%%%%%%%%%%%%%%%%%%%%%%%%
%%%%%%%%%%%%%%%%%%%%%%%%%%%%%%%%%%%%%%%%%%%%%%%%%%%%%

\qquad\\
\qquad\\
\noindent
FABIO SCARABOTTI, Dipartimento SBAI, Universit\`a degli Studi di Roma ``La Sapienza'', via A. Scarpa 8, 00161 Roma (Italy)\\
{\it e-mail:} {\tt scarabot@sbai.uniroma1.it}\\

\noindent
FILIPPO TOLLI, Dipartimento di Matematica e Fisica, Universit\`a Roma TRE, L. San Leonardo Murialdo 1, 00146 Roma, Italy
{\it e-mail:} {\tt tolli@mat.uniroma3.it}\\


\begin{thebibliography}{99}
\bibitem{Bump} D. Bump, {\it Lie groups}. Graduate Texts in Mathematics, 225. Springer-Verlag, New York, 2004.
\bibitem{tree}  T. Ceccherini-Silberstein, F. Scarabotti and F. Tolli,  Trees, wreath products and finite Gelfand pairs, {\it Adv. Math.}, {\bf 206} (2006), no. 2, 503--537.
\bibitem{GelGeorg}T. Ceccherini-Silberstein,  F. Scarabotti and F. Tolli, Finite Gelfand pairs and their applications to probability and statistics, {\it  J. Math. Sci. (N. Y.)} {\bf 141} (2007), no. 2, 1182--1229.
\bibitem{book} T. Ceccherini-Silberstein, F. Scarabotti and F. Tolli, {\it Harmonic analysis on finite groups:
representation theory, Gelfand pairs and Markov chains.}  
Cambridge Studies in Advanced Mathematics { 108}, Cambridge University Press 2008.
\bibitem{Mackey} T. Ceccherini-Silberstein, A. Mach\`i, F. Scarabotti, F. Tolli,  Induced representations and Mackey theory, Functional analysis, {\it J. Math. Sci. (N. Y.)} {\bf 156} (2009), no. 1, 11--28.
\bibitem{Clifford}T. Ceccherini-Silberstein, F. Scarabotti and F. Tolli, Clifford theory and applications,
 Functional analysis, {\it J. Math. Sci. (N. Y.)} {\bf 156} (2009), no. 1, 29--43.
\bibitem{wreath} T. Ceccherini-Silberstein, F. Scarabotti and F. Tolli, Representation theory of wreath
products of finite groups,  Functional analysis, {\it J. Math. Sci. (N. Y.)} {\bf 156} (2009), no. 1, 44--55.
\bibitem{book2} T. Ceccherini-Silberstein, F. Scarabotti and F. Tolli, {\it Representation theory of the symmetric groups: the Okounkov-Vershik approach, character formulas, and partition algebras.}  
Cambridge Studies in Advanced Mathematics { 121}, Cambridge University Press 2010.
\bibitem{CURFOS} C.~W. Curtis, T.~V Fossum, On centralizer rings and characters of representations of finite groups, {\it Math. Z.} {\bf 107} (1968) 402--406.
%\bibitem{CR1} Ch. W. Curtis and I. Reiner, {\it Representation theory of finite groups and associative algebras}. Reprint of the 1962 original. Wiley Classics Library. A Wiley-Interscience Publication. John Wiley \& Sons, Inc., New York, 1988. 
\bibitem{CR2} Ch.~ W. Curtis and I. Reiner,{ \it  Methods of Representation Theory. With Applications to Finite Groups
and Orders}. Vols. I, Pure Appl. Math., John Wiley \& Sons, New York 1981.
\bibitem{DD1} D. D'Angeli and A. Donno, Crested products of Markov chains, {\it  Ann. Appl. Probab.}, {\bf  19}, no. 1, (2009),
414--453.
\bibitem{DD2} D. D'Angeli and A. Donno, Markov chains on orthogonal block structures, {\it  European J. Combin.}, {\bf  31},
Issue 1 (2010), 34--46.
\bibitem{DD3} D. D'Angeli and A. Donno, Generalized crested products, {\it  European J. Combin.}, {\bf  32}, Issue 2 (2011), 243--257.
\bibitem{Diaconis} P. Diaconis, {\it Groups Representations in Probability and Statistics.} IMS Hayward, CA, 1988.
\bibitem{Faraut} J. Faraut,  Analyse harmonique sur les paires de Gelfand et les espaces hyperboliques, in {\it Analyse harmonique}, J.~L. Clerc, P. Eymard, J. Faraut,  M.  Ra\'es, R. Takahasi, Eds (\'Ecole d'\'et\'e d'analyse harmonique Universit\'e de Nancy I), C.I.M.P.A. V, 1983.
\bibitem{FN} A. Fig\`a-Talamanca, C. Nebbia, {\it Harmonic Analysis and Representation Theory for Groups Acting on Homogeneous Trees.} London Mathematical Society Lecture Notes Series, 162. Cambridge University Press, Cambridge, 1991.
%\bibitem{Janusz} J.~G. Janusz, Primitive idempotents in group algebras, {\it Proc. Amer. Math. Soc.} {\bf 17} (1966) 520--523.
\bibitem{Macdonald} I.~G. Macdonald, {\it Symmetric functions and Hall Polynomials}, second edition. Oxford University Press, 1995.
\bibitem{MIZ000} H. Mizukawa, Zonal spherical functions on the complex reflection groups and (n+1,m+1)-hypergeometric
functions, {\it  Adv. Math.} {\bf 184} (2004) 1--17.
\bibitem{Mizukawa} H. Mizukawa, {Twisted Gelfand pairs of complex reflection groups and $ r$-congruence properties of Schur functions}, {\it Ann. Comb.} {\bf 15} (2011), no. 1,109--125.
\bibitem{NS} M.A. Naimark and A.I. Stern, {\it Theory of Group Representations.} Springer-Verlag, New York, 1982.
\bibitem{Nebe} G. Nebe, Orthogonal Frobenius reciprocity, {\it  J. Algebra} {\bf 225} (2000), no. 1, 250--260.
\bibitem{OV} A. Okounkov and A.~M. Vershik,  A new approach to the representation theory of  symmetric groups. II (Russian),
{\it Zap.Nauchn. Sem. S.Petersburg. Otdel. Mat. Inst. Steklov.} (POMI) {\bf 307} (2004), Teor. Predst. Din. Sist. Komb. i Algoritm. Metody. 10,  57--98, 281;  translation in {\it J. Math. Sci.}  (N.Y.) {\bf 131} (2005), no. 2, 5471--5494.
%\bibitem{PS} I. Piatetski-Shapiro, {\it Complex representations of GL(2,K)  for finite fields K.} Contemporary Mathematics, 16. American Mathematical Society, Providence, R.I., 1983.
\bibitem{ScarabottiLapBer} F. Scarabotti, Time to reach stationarity in the Bernoulli-Laplace diffusion model
with many urns, {\it Adv. in Appl. Math.} {\bf 18} (1997), no. 3, 351--371.
\bibitem{st1}      F. Scarabotti and F. Tolli, Harmonic analysis on a finite homogeneous space, {\it Proc. Lond. Math. Soc.}(3) {\bf{100}} (2010),  no. 2,
348--376.
\bibitem{st2}      F. Scarabotti and F. Tolli, Harmonic analysis on a finite homogeneous space II: the Gelfand Tsetlin decomposition,  {\it 
Forum Mathematicum}  {\bf{22}} (2010), 897--911.
\bibitem{st3}  F. Scarabotti and F. Tolli, Fourier analysis of subgroup-conjugacy invariant functions on finite groups, to appear in {\it Monoshafte fur Mathematik}.
\bibitem{st4}  F. Scarabotti and F. Tolli, Hecke algebras and harmonic analysis on  finite groups, to appear in {\it Rendiconti di Matematica}.
\bibitem{Stembridge}  J.~R. {Stembridge}, On Schur's Q-functions and the primitive idempotents of a commutative Hecke algebra, {\it J. Algebr. Comb.} {\bf 1} (1992) 71--95.
\bibitem{Sternberg} S. Sternberg, {\it Group theory and physics.} Cambridge University Press, Cambridge, 1994.
\bibitem{Terras} A. Terras, {\it Fourier analysis on finite groups and applications}.
London Mathematical Society Student Texts, 43. Cambridge University Press, Cambridge, 1999.
\end{thebibliography}
\end{document}